\documentclass[11pt,twoside,reqno]{article}

\usepackage{amssymb}
\usepackage{mathrsfs}
\usepackage{amsfonts}
\usepackage{graphicx}
\usepackage{mathptmx}
\usepackage{latexsym}
\usepackage{color}
\usepackage{amssymb}
\usepackage{amsmath}
\usepackage{amsthm}
\usepackage{verbatim}

\allowdisplaybreaks

\renewcommand{\r}{\right}
\renewcommand{\l}{\left}

\def\vd{\mathrm{d}}

\def\dsum{\displaystyle\sum}
\def\dint{\displaystyle\int}

\def\dlim{\displaystyle\lim}

\newtheorem{theorem}{Theorem}[section]
\newtheorem{lemma}[theorem]{Lemma}
\newtheorem{corollary}[theorem]{Corollary}
\newtheorem{proposition}[theorem]{Proposition}
\theoremstyle{definition}
\newtheorem{remark}[theorem]{Remark}

\numberwithin{equation}{section}

\begin{document}

\arraycolsep=1pt

\title{\Large\bf  Reflecting Diffusion Semigroup on  Manifolds carrying Geometric Flow
\footnotetext{\hspace{-0.35cm} 2010 {\it Mathematics Subject
Classification}. {Primary 58J65; Secondary  60J60, 53C44.}
\endgraf{\it Key words and phrases}. {Geometric flow,  Ricci flow, curvature, second fundamental form, transportation-cost inequality, Harnack inequality, coupling.}
\endgraf{\it *} {Corresponding author}.
\endgraf{\it *}{Department of Applied Mathematics, Zhejiang University of Technology, Hangzhou 310023, China}.
\endgraf{$\sharp$}{Department of Management Sciences, City University of Hong Kong, Kowloon, Hong Kong 999077, China}.
\endgraf{\it E-mail}:\texttt{chenglj@zjut.edu.cn}(L.J.Cheng), \texttt{zhangkunbnu05@163.com}(K.Zhang)}
}
\author{{\bf Li-Juan Cheng, $^{\ast}$  \ \ \ \  \  Kun Zhang $^{\sharp}$}}
\date{ }
\maketitle

\vspace{-0.6cm}

\begin{center}
\begin{minipage}{13.5cm}\small
{{\bf Abstract.} Let $L_t:=\Delta_t+Z_t$ for a $C^{\infty}$-vector field $Z$ on a  differentiable manifold $M$ with boundary $\partial M$,
where $\Delta_t$ is the Laplacian operator, induced by a time dependent metric $g_t$ differentiable in $t\in [0,T_c)$.
We first establish the derivative formula for the associated reflecting diffusion semigroup  generated by $L_t$; then construct
 the couplings for the reflecting $L_t$-diffusion processes  by parallel displacement and reflection, which are applied to   gradient estimates and Harnack inequalities  of the associated heat semigroup; and finally, by using the derivative formula, we present a number of equivalent inequalities for a new curvature lower bound  and the  convexity of the boundary, including the  gradient estimates, Harnack inequalities,  transportation-cost inequalities and other functional inequalities   for  diffusion semigroups.}
\end{minipage}
\end{center}

\section{Introduction and main results}
\quad
It is well known that,  functional inequalities, for instance, gradient inequalities and dimension-free Harnack inequalities, are useful tools on stochastic analysis  to investigate the behavior of the underlying processes on Riemannian manifolds, see, for example, \cite{CW94, CW97a, EL, Hsu, Wbook2}.
Among all those work,  one usually makes the assumption that the metric is fixed. However,
when it comes to the case that metric is time-varying,  a question arises naturally:
how about functional inequalities on these manifolds?
In recent year, M. Arnaudon, K. Coulibaly and A.  Thalmaier \cite{ACT}
  constructed  $g_t$-Brownian motions (i.e., the diffusion process generated by $L_t=\frac{1}{2}\Delta_t$) on manifolds without boundary carrying
a geometric flow, and established the Bismut
formula under the  Ricci flow, which in particular implies the gradient estimates
of the associated heat semigroup.  In \cite{cheng}, the first author studied functional inequalities,
including those
on manifolds carrying geometric flow for the  diffusion semigroup.
 Motivated by the aforementioned results, this article aim to extends these results in \cite{ACT, cheng} to the case with boundary.

The setting for our work is a differentiable  manifold with boundary
equipped with a geometric flow. More precisely,  let $M$  be a $d$-dimensional  differentiable manifold with boundary $\partial M$, which carries   a
one-parameter $C^{\infty}$-family of complete Riemannian metrics $\{g_t\}_{t\in [0,T_c)}$, where $T_c$ is the time when the curvature may blow up.
Consider the elliptic operator $L_t:=\Delta_t+Z_t$, where $\Delta_t$ is the  Laplacian operator associated with the metric $g_t$ and $(Z_t)_{t\in [0,T_c)}$ is a $C^{\infty}$-family of vector fields.  Let  $(X_t)$ be  a reflecting
diffusion process generated by $L_t$ (called the reflecting $L_t$-diffusion process), which is assumed to be non-explosive. This  assumption immediately implies that this process corresponds in a natural way to a strongly continuous semigroup $P_{s,t}$, i.e.,
$$P_{s,t}f(x)=\mathbb{E}(f(X_t)|X_s=x),\quad 0\leq s\leq t<T_c.$$
In this article, we   fix on extending  the former discussions to our content for the semigroup $P_{s,t}$. Compared with
F.-Y. Wang's work on  functional inequalities over Riemannian manifolds  with boundary (see for example \cite{W09a,WSe, W10,WPC,WAP, Wbook2} and the reference therein),  we need to make some necessary modifications to our inhomogeneous context, since e.g.  geometric quantities are time-dependent and the underlying process is time-inhomogeneous.

Before moving on, let us briefly recall some known results in the time-inhomogeneous Riemannian setting without boundary. M. Arnoudon,
 K. Coulibaly and A. Thalmaier \cite{ACT1} investigated  the
optimal transportation inequality by constructing horizontal diffusion processes. Then, K. Kuwada and R. Philipowski \cite{Ku} studied  the non-explosion of $g_t$-Brownian
motions   under the super Ricci flow, and K. Kuwada \cite{Ku3} developed the coupling method to estimate the gradient of the semigroup. Very recently, the author \cite{cheng} has
considered the construction of coupling processes and some important functional inequalities on manifolds without boundary carrying a geometric flow.
All these works lay solid foundation for our study.

Let $\nabla^t$  be  the Levi-Civita connection  associated with the
metric $g_t$.
For simplicity, we introduce the notation: for $X, Y\in TM$,
\begin{align*}
    &\mathcal{R}_t^{Z}(X,Y):={\rm Ric}_t(X,Y)-\l<\nabla^t_XZ_t, Y\r>_t-\frac{1}{2}\partial_tg_t(X,Y),
\end{align*}
 where ${\rm Ric}_t$ is the Ricci curvature tensor with respect to the metric $g_t$,  and $\l<\cdot,\cdot\r>_t:=g_t(\cdot,\cdot)$. Define the second fundamental form of the boundary with respect to $g_t$ by
 $${\rm II}_{t}(X,Y)=-\l<\nabla^t_XN_t,Y\r>_t, \ \ X,Y\in T\partial M,$$
 where  $N_t$ is the
inward unit normal vector field of the boundary associated with the metric $g_t$.
If ${\rm II}_t\geq 0$ for all $t\in [0,T_c)$,  then the geometric flow $\{g_t\}_{t\in [0,T_c)}$ is called to be convex. In the fixed metric case, many functional inequalities are always deduced under the Bakry-Emery curvature condition.  In this paper,  we begin our discussion  by using the following  curvature constraints:
\begin{equation}\label{CV1-1}
\mathcal{R}_t^Z\geq K(t,\cdot)\quad\mbox{and}\quad {\rm II}_t\geq \sigma(t,\cdot)
\end{equation}
for some continuous functions $K,\sigma \in C([0,T_c)\times M)$. Here and in what follows, for any two-tensor $\mathbf{T}_t$ and any function $f$, we write $\mathbf{T}_t\geq f$ if ${\bf T}_t(X,X)\geq f\l<X,X\r>_t$, for $X\in TM$ and $t\in [0,T_c)$.
 Compared with  the usual Bakry-Emery curvature condition, the time derivative of the metric will become a new important
term the curvature condition.

Let $\rho_t$ be the Riemannian distance  and $|\cdot|_t$ be the norm associated with the metric $g_t$. When the geometric  flow is convex, we have the first main result of this paper.
\begin{theorem}\label{main-th-1}
For any $K\in C([0,T_c))$,  the following statements are equivalent to each other.

\begin{itemize}
 \item [$(i)$] The following curvature condition holds,
 \begin{equation}\label{CV1}
   \mathcal{R}^Z_t\geq K(t)\ \  \mbox{and}\ \    {\rm II}_t\geq 0~(\partial M\neq \varnothing)\ \  \mbox{for\ all} \ \  t\in [0,T_c).
 \end{equation}
  \item [$(ii)$]  The gradient inequality
\begin{align}\label{add-1}
|\nabla ^sP_{s,t}f|_s\leq e^{-\int_s^t K(r)\vd r}P_{s,t}|\nabla^tf|_t,\ 0\leq s\leq t<T_c
\end{align}
holds for $f\in C^1(M)$ such that $f$ is constant outside a compact set of $M$.\vspace{-0.2cm}
  \item [$(iii)$]  For any $p>1$, $0\leq s<t<T_c$ and $f\in \mathscr{B}_b^+(M)$,
  \begin{align}\label{add-2}
  (P_{s,t}f)^p(x)\leq P_{s,t}f^p(y)\exp{\l[\frac{p}{4(p-1)}{\l(\int_s^te^{2\int_s^rK(u)\vd u}\vd
r\r)}^{-1}\rho_s^2(x,y)\r]}.
\end{align}
\end{itemize}
\end{theorem}
We intend to use the coupling method to prove that \eqref{CV1} implies \eqref{add-1} and \eqref{add-2}. It is well known that coupling method is a useful tool in stochastic analysis. It is remarkable that M.-F. Chen and F.-Y. Wang \cite{CW94,CW97a,CW97b} gave subtle estimates about the first eigenvalue on Riemannian manifolds by constructing suitable coupling processes.  Note that K. Kuwada \cite{Ku3} first constructed  the coupling processes for $L_t$-diffusion processes on manifolds without boundary via
discrete approximation. In our recent work \cite{cheng}, we gave a direct construction for general  coupling processes on manifolds without boundary.
Here, we modify this proof to our setting.


On the other hand,  to prove that each of \eqref{add-1}
and \eqref{add-2} implies the curvature condition \eqref{CV1},  we need to use the derivative formula to characterize $\mathcal{R}_t^Z$ and ${\rm II}_t$ first. In the following section, we  construct a series of Hsu's multiplicative functionals to establish the derivative formula (see Theorem \ref{Bis} below). When the metric is independent of $t$, our construction is due to  \cite[Theorem 3.2.1]{Wbook2} for the constant metric case.
In fact, it is more difficult for us  to deal with the case of differentiable manifolds carrying the non-convex flow,
 since it is hard to control the effect from the boundary by using the coupling method.
 A direct thought is to make a conformal change of the metrics such that the  new flow becomes convex.
When the metric is independent of $t$, this method is successfully applied to the non-convex manifold,
see \cite{W07,W10,WAP}.  First, let us introduce an important set:
\begin{equation}\label{D}
   \mathscr{D}=\{\phi\in C^{1,\infty}([0,T_c)\times M): \inf \phi_t=1,  \ {\rm II}_t\geq -N_t\log \phi_t\}.
\end{equation}
Then, by \cite[Lemma 2.1]{W07}, for $\phi\in \mathscr{D}$, the new flow  $\tilde{g}_t:=\phi_t^{-2}g_t$ is   convex. Moreover, we are required to having the following assumption on $\phi$, ${\rm Ric}_t^{Z}$,  and $ \partial_tg_t$ to continue our discussion.
\begin{description}
  \item[(H1)] Let $d\geq 2$. There exist functions $K_1, K_2\in C([0,T_c))$ such that
\begin{align}
    {\rm Ric}_t^{Z}:={\rm Ric}_t-\nabla^tZ_t\geq K_1(t), \qquad
     \partial_tg_t\leq K_2(t), \label{CV3}
\end{align}
and $\phi\in\mathscr{D}$
such that $\|\nabla^t \phi_t\|_{\infty}<\infty$, $\|\phi_t\|_{\infty}<\infty$ and
\begin{align*}
K_{\phi,1}(t):&=\inf_{M}\l\{\phi_t^2K_1(t)+\frac{1}{2}L_t\phi_t^2-|\nabla^t\phi_t^2|_t\cdot|Z_t|_t-(d-2)|\nabla^t\phi_t|^2_t\r\}>-\infty,\\
K_{\phi,2}(t):&=\sup_{M}\{2\partial _t\log\phi_t\}+K_2(t)<\infty,
\end{align*}
where $\|\nabla^tf\|_{\infty}:=\sup_{x\in M}|\nabla^tf|_t(x)$.
\end{description}
If this assumption holds, then by constructing suitable coupling processes, we have the second main result of this paper.
\begin{theorem}\label{P1}
Suppose that  {\bf (H1)} holds and
$$K_{\phi}(t):=K_{\phi,1}^-(t)+\frac{1}{2}K_{\phi,2}(t)+[2\|\phi_tZ_t+(d-2)\nabla^t\phi_t\|_{\infty}+d\|\nabla^t\phi_t\|_{\infty}]\|\nabla^t\phi_t\|_{\infty}<\infty.$$
Then the following conclusions hold.
\begin{itemize}
         \item [$(i)$] For any $f\in C^1(M)$ such that $f$ is constant outside a compact set,
$$|\nabla^sP_{s,t}f|_s\leq \|\phi_t\|_{\infty}\|\nabla^tf\|_{\infty}e^{\int_s^tK_{\phi}(r)\vd r},\ \ 0\leq s\leq t< T_c.$$
         \item [$(ii)$] For  $0\leq s<t<T_c$, let $\delta_{s,t}=1-\sup_{r\in [s,t]}\|\phi_r\|_{\infty}^{-1}$, $\lambda_{s,t}=\inf_{(r,x)\in [s,t]\times M}\phi^{-1}$ and $$\delta_{p}=\max\l\{\delta_{s,t},\frac{\lambda_{s,t}}{2}(\sqrt{p}-1)\r\}.$$ Then
for $p>(1+\frac{\delta_{s,t}}{\lambda_{s,t}})^2$, $x,y\in M$ and  $f\in C_b(M)$,
   it holds
  $$(P_{s,t}f(y))^p\leq P_{s,t}f^p(x)\exp\l\{\frac{\sqrt{p}(\sqrt{p}-1)\rho_s(x,y)}{8\delta_p[(\sqrt{p}-1)\lambda_{s,t}-\delta_p]\int_s^te^{-2\int_s^r(K_{\phi}(u)+\|\nabla^u\phi_u\|_{\infty}^2)\vd u}\vd r}\r\}.$$

       \end{itemize}
\end{theorem}

As an important  application of the induced conclusions above for general geometric flow,  we  consider the Ricci flow with umblic boundary as follows:
 for $\lambda\geq 0$,
 \begin{equation}\label{Ricci-flow}
    \begin{cases}
\frac{\partial}{\partial t}g(x,\cdot)(t)=2{\rm Ric}(x,t), \  & \  \ (x,t)\in M\times [0,T];\\
\\
{\rm II}(x,t)=\lambda g(x,t),\ & \ \ x\in \partial M.
\end{cases}
 \end{equation}
 Shen \cite{Shen} proved  the short time existence of the solution to the above equation. We also refer the reader to \cite{BCP} for  more  geometric explanation for this Ricci flow. To our knowledge,  there are few references about gradient estimates and Harnack inequalities for the solution to the heat equation under the Ricci flow carrying non-convex umbilic boundary.    In Section 3.3,   we will apply Theorems \ref{main-th-1} and \ref{P1} to establish  these inequalities for this  system;  see Theorems \ref{Ricci-gradient} and  \ref{Ricci-Harnack-inequality} below.

The rest parts of the paper are
organized as follows. In Section 2, we construct the reflecting $L_t$-diffusion processes,  prove
the Kolmogorov equations  and then establish the derivative formula for the associated semigroup. In
Sections 3, we turn to prove Theorems \ref{main-th-1} and \ref{P1} by constructing the coupling processes, which are applied to  the Ricci flow with umbilic boundary. In Section 4,  some important inequalities including transportation-cost inequality, Harnack inequalities and other functional inequalities are proved to be equivalent to the lower bound of $\mathcal{R}^Z_t$ and the convexity of the boundary.

We end this section by making some conventions on the notations.
  Let $\mathscr{B}_b(M)$ be the set of all measurable functions and $C^p_0(M)$ the set of all $C^p$-smooth real functions with compact supports on $M$.
For any function  $f$ and $\varphi$ respectively defined on $[0,T_c)\times M$  and  $[0,T_c)\times M\times M$, we simply write $f_t(x):=f(t,x)$ and $\varphi_t(x,y):=\varphi(t,x,y)$, $t\in [0,T_c), x,y\in M$. In addition, $\|f_t\|_{\infty}:=\sup_{x\in M}f(t,x)$ and $\|f\|_{\infty}=\sup_{(t,x)\in [0,T_c)\times M}f(t,x)$.  For any time-depending vector field $V_t$, we write $\|V_t\|_{\infty}:=\||V_t|_t\|_{\infty}$ for simplicity.

\section{Preliminaries}
\hspace{0.5cm} In Subsection 2.1, we briefly introduce the construction of reflecting $L_t$-diffusion
processes. In Subsection 2.2, the forward and backward Kolmogorov equations are established for Neumann diffusion semigroup. In Subsection 2.3, a derivative formula is established, which is further applied to characterizing $\mathcal{R}^Z_t$ and ${\rm II}_t$.
\subsection{Reflecting $L_t$-diffusion processes}
\hspace{0.5cm} Let $\mathcal{F}(M)$  be the frame bundle over $M$ and
$\mathcal{O}_{t}(M)$
the orthonormal frame bundle over $M$ with respect to the metric  $g_t$.  Set $\mathbf{p}:
\mathcal{F}(M)\rightarrow M$ be the projection
from $\mathcal{F}(M)$ onto $M$.
Let $\{e_{i}\}_{i=1}^{d}$ be the canonical orthonormal basis of
$\mathbb{R}^d$.
 For any $u\in \mathcal{O}_t(M)$,
let $H^{t}_{X}(u)$ be the $\nabla^{t} $-horizontal lift
of $X\in T_{{\bf p}u}M$ and $H_{i}^t(u)=H_{ue_i}^t(u), i=1,2,\cdots, d$.
 For any $u\in \mathcal{F}(M)$,
let $\{V_{\alpha, \beta}(u)\}_{\alpha,\beta=1}^d$ be the canonical basis
of vertical  fields over $\mathcal{F}(M)$.

 Let $B_t:=(B_t^1,B_t^2,\cdots,B_t^d)$ be a $\mathbb{R}^d$-valued Brownian motion on  a complete
filtered probability space $(\Omega,\{\mathscr{F}_t\}_{t\geq 0}, \mathbb{P})$  with the natural filtration
$\{\mathscr{F}_t\}_{t\geq 0}$.
As in the time-homogeneous case, to construct the reflecting $L_t$-diffusion process,
we first construct the corresponding horizontal diffusion process
  by solving the Stratonovich stochastic
diffusion equation (SDE):
$$\begin{cases}
 \vd u_t=\sqrt{2}\dsum_{i=1}^{d}H_{i}^t(u_t)\circ \vd
B_t^{i}+H_{Z_t}^t(u_t)\vd t-\frac{1}{2}\dsum_{i,j}\partial_tg_t(u_te_i,u_te_j)V_{i, j}(u_t)\vd t+H^t_{N_t}(u_t)\vd l_t,\medskip\\
u_0\in \mathcal{O}_0(M),\ {\bf p}u_0=x\in M,
\end{cases}$$
where 
$l_t$ is an increasing
process supported on $\{t\in[0,\zeta):X_t:={\bf p}u_t\in \partial M\}$, where $\zeta:=\dlim_{n\rightarrow
\infty}\zeta_n$ and
$$\zeta_n:=\inf\{t\in [0,T_c):\rho_t({\bf p}u_0, {\bf p}u_t)\geq n\}, \ n\geq 1,\  \inf\varnothing=T_c.$$  Similarly as explained in \cite{ACT}, the last term is essential to ensure $u_t\in \mathcal{O}_t(M)$.
Then, it is easy to see that $X_t:={\bf p}u_t$
solves the equation $$\vd X_t=\sqrt{2}u_t\circ \vd B_t+Z_t(X_t)\vd
t+N_t(X_t)\vd l_t,\ X_0=x $$ up to the life time $\zeta$.
By the It\^{o} formula, for any $f\in C_0^{1,2}([0,T_c)\times M)$ with
$N_tf_t:=N_tf(t,\cdot)|_{\partial M}=0$ (i.e. $N_tf_t$ is taken value automatically on $\partial M$),
$$f(t,X_t)-f(0,x)-\int_0^t\l({\partial_ s}+L_s\r)f(s,X_s)\vd
s=\sqrt{2}\int_0^t\l<u_s^{-1}\nabla^s f(s,\cdot)(X_s),\vd B_s\r>_s$$
is a
martingale up to the life time $\zeta$. So, we call $X_t$  the
reflecting diffusion process generated by $L_t$.

Throughout this paper, we only consider the case where the reflecting $L_t$-diffusion process
is non-explosive before $T_c$.   In this case,
$$P_{s,t}f(x):=\mathbb{E}(f(X_t)|X_s=x),\ x\in M,\ 0\leq s\leq t< T_c,\ f\in \mathscr{B}_b(M)$$
gives rise to a Markov  semigroup $\{P_{s,t}\}_{0\leq
s\leq t< T_c}$ on $\mathscr{B}_b(M)$, which is called the
Neumann semigroup generated by $L_t$.
Here and in what follows,
$\mathbb{E}$ and $\mathbb{P}$ (resp. $\mathbb{E}^x$ and $\mathbb{P}^x$) stand for the expectation
and probability taken for the underlying process (resp. the underlying process starting from $x\in M$).

\subsection{Kolmogorov equations}
Let
$$\mathscr{C}_N(L)=\{f\in C^{1,\infty}([0,T_c)\times M): N_tf_t|_{\partial M}=0, (L_t+\partial_t)f\in \mathscr{B}_b(M),\, t\in [0,T_c)\}.$$
 In this subsection, we now introduce the Kolmogorov equations for $P_{s,t}$ as follows.
 \begin{theorem}\label{2t1}
 For $f\in \mathscr{C}_N(L)$,  the following forward Kolmogorov equation holds,
\begin{align}\label{fKol}
   \frac{\partial}{\partial t}P_{s,t}f(t,x)=P_{s,t}\l(L_tf+{\partial_
   t}f\r)(t,x), \ \ 0\leq s<t<T_c.
   \end{align}
  Moreover, for $f\in \mathscr{C}_N(L)$, then the following (i) and (ii) hold
 \begin{itemize}
   \item [(i)]  for any $0\leq t<T_c$,  $P_{\cdot,t}f\in C^{1,2}([0,t]\times M)$ and the backward Kolmogorov equation
   \begin{align}\label{eq-2.2}
   \frac{\partial }{\partial s}P_{s,t}f=-L_sP_{s,t}f,\hspace{0.5cm}  0\leq s<t<T_c,
   \end{align}
 moreover,  $$N_sP_{s,t}f=0, \hspace{0.5cm} 0\leq s<t<T_c;$$
   \item [(ii)]   if
   $|\nabla^{\cdot}P_{\cdot, t}f|_{\cdot}$ is bounded on $[r,t]\times
   M$ and $t\in (0,T_c]$, then
   $$\frac{\partial }{\partial s}P_{r,s}\psi(P_{s,t}f)=P_{r,s}\l(\psi''(P_{s,t}f)|\nabla^sP_{s,t}f|_s^2\r),\  \  s\in[r,t],$$
   where $\psi\in C^2(\mathbb{R})$ with compact support in $[\inf f,\sup f]$.
 \end{itemize}
 \end{theorem}
To prove this theorem, we need the following two properties to  investigate the short time behavior of  the diffusion process first.
\begin{proposition}\label{2l1}
Let $X_t$ be a reflecting $L_t$-diffusion
process with $X_0=x\in M$, Then,
\begin{enumerate}
\item [$(a)$] if $x\in M^{\circ}$, then for $t_0\in [0,T_c)$, there exist
constants $r_0>0$ and $c_1>0$ such that $B_{t_0}(x,r_0)\in M^{\circ}$ and
$$\mathbb{P}^x(\sigma_r\leq t)\leq c_1e^{-r^2/16t},\ r\in [0,r_0],\ t\in [0,1\wedge T_c]$$
holds, where $\sigma_r:=\inf\{s: \rho_{t_0}(X_s,x)\geq r,\ s\in [0,T_c)\}$;
 \item [$(b)$] there exist
constants $r_0>0$ and  $c_2>0$ such that
$$\mathbb{P}^x(\tilde{\sigma}_r\leq t)\leq c_2e^{-r^2/16t},\ r\in [0,r_0],\ t\in [0,1\wedge T_c]$$
holds, where $\tilde{\sigma}_r:=\inf\{s:
\rho_s(X_s,x)\geq r,\ s\in [0,T_c)\}$.
 \end{enumerate}
\end{proposition}
\begin{proof}
 First, we prove (a). 
  Write ${\rho}_{t_0}(X_t):={\rho}_{t_0}(x,X_t)$ for simplicity.
By taking smaller $r_0$, we may and do assume that $B_{t_0}(x,r_0)\in M^{\circ}$ and ${\rho}_{t_0}\in C^{\infty}(M)$.
By  the It\^{o} formula, we obtain
\begin{align*}
\vd {{\rho}}_{t_0}^2(X_t)\leq  2\sqrt{2}{{\rho}}_{t_0}(X_t)\vd b_t+C_1\vd t,\ \ t\leq \sigma_{r}
\end{align*}
for some constant $C_1>0$,
where $b_t$ is a one-dimensional Brownian motion. Thus, for fixed $t>0$ and $\delta>0$,
$$Z_s:=\exp\l(\frac{\delta}{t}{\rho}_{t_0}(X_s)^2-\frac{\delta}{t}C_1s-4\frac{\delta^2}{t^2}\int_0^s {\rho}_{t_0}(X_u)^2\vd u\r),\qquad 0 \leq s\leq \sigma_{r}$$
is a supermartingale. Therefore,
\begin{align}
\mathbb{P}^x(\sigma_{r}\leq t)&=\mathbb{P}^x\l\{\max_{s\in [0,t]}{\rho}_{t_0}(X_{s\wedge\sigma_{r}})\geq r\r\}\leq \mathbb{P}^x\l\{\max_{s\in [0,t]}Z_{s\wedge \sigma_{r}}\geq e^{\delta r^2/t-\delta C_1-4\delta^2r^2/t}\r\}\nonumber\\
&\leq \exp{\l[C_1\delta-\frac{1}{t}(\delta r^2-4\delta^2r^2)\r]}.
\end{align}
The proof of (a) is completed by taking $\delta:=1/8$.

We now turn to prove (b).  Let $\phi \in
C^{1,\infty}([0,1]\times M)$ be constant outside ${\bf B}=\{(t,y)\in [0,1]\times M: \rho_t(x,y)\leq r_0 \}$  such that
$\phi\geq 1$ in ${\bf B}$, and the boundary  $\partial M$ in $B_{t}(x,r_0)$ is convex under $\tilde{g}_t:=\phi_t^{-2}g_t$ (For fixed time $t\in [0,T_c)$, see \cite{W07} for the
construction of $\phi_t=\varphi \circ \rho_t^{\partial}$, where $\varphi$ is differentiable on $[0,\infty)$. As $\rho_t^{\partial}$ is differentiable in $t$, we have that $\phi \in C^{1,\infty}([0,1]\times M)$).
Let $\tilde{\Delta}_{t}$ and $\tilde{\nabla}^t$ be respectively the Laplacian and the gradient operators induced by the metric $\tilde{g}_t$. Then, we have
$$\phi_t^2L_t=\tilde{\Delta}_t+(d-1)\phi_t\nabla^t \phi_t+\phi_t^2Z_t=: \tilde{\Delta}_t+\tilde{Z}_t,$$
and $X_t$ solves the SDE:
\begin{align}\label{SDE-phi}
\vd_{ I}X_t=\sqrt{2}\phi_t^{-1}u_t\vd B_t+\phi_t^{-2}\tilde{Z}_t(X_t)\vd t+\tilde{N}_t(X_t)\vd l_t,
\end{align}
where $\tilde{N}_t$ is the inward unit normal vector field of the boundary associated with the metric $\tilde{g}_t$ and $\vd_I$ denotes the It\^{o} differential\footnote{In local
coordinates, the It\^{o} differential for a continuous
semi-martingale $X_t$ on $M$ is given by (see e.g. \cite{Emery})
$$(\vd_I X_t)^k=\vd X^k_t+\frac{1}{2}\sum_{i,j=1}^d\Gamma_{i,j}^{k}(t,X_t)\vd \l<X^i, X^j\r>_t,\ \ 1\leq k\leq d,$$
where $\Gamma^k_{ij}(t,x)$ are the Christoffel symbols with respect to the metric $g_t$.} on $M$.
Let $\tilde{\rho}_t$ be the Riemannian distance  under the  metric $\tilde{g}_t$. By taking smaller $r_0$, we may and do assume that $\tilde{\rho}^2\in C^{1,\infty}({\bf B})$. Then,  we have that  there exists a constant $C_1>0$ such that
 $$(\partial_t+ L_t)\tilde{\rho}^2_t(x,\cdot)(y)=2\tilde{\rho}_t(x,y)L_t\tilde{\rho}_t(x,\cdot)(y)+2\tilde{\rho}_t(x,y)\partial_t\tilde{\rho}_t(x,y)+2|\nabla^t \tilde{\rho}_t|^2_t\leq C_1$$
 holds on ${\bf B}$. By the It\^{o} formula, for $r\in (0,r_0)$, we further obtain
 $$\vd \tilde{\rho}^2_t(x,X_t)\leq 2\sqrt{2}\phi^{-1}_t\tilde{\rho}_t(x,X_t)\vd b_t+C_1 \vd t,\ \ 0 \leq t\leq \tilde{\sigma}_{r},$$
where $b_t$ is a one-dimensional Brownian motion. Let $\tilde{\rho}_t(X_t):=\tilde{\rho}_t(x,X_t)$. Thus, for fixed $t>0$ and $\delta>0$,
$$Z_s:=\exp\l(\frac{\delta}{t}\tilde{\rho}_{s}(X_s)^2-\frac{\delta}{t}C_1s-4\frac{\delta^2}{t^2}\int_0^s \tilde{\rho}_{u}(X_u)^2\vd u\r),\qquad 0 \leq s\leq \tilde{\sigma}_{r}$$
is a supermartingale. Therefore,
\begin{align*}
\mathbb{P}^x(\tilde{\sigma}_{r}\leq t)&=\mathbb{P}^x\l\{\max_{s\in [0,t]}\tilde{\rho}_{s}(X_{s\wedge\tilde{\sigma}_{r}})\geq r\r\}\leq \mathbb{P}^x\l\{\max_{s\in [0,t]}Z_{s\wedge \tilde{\sigma}_{r}}\geq e^{\delta r^2/t-\delta C_1-4\delta^2r^2/t}\r\}\nonumber\\
&\leq \exp{\l[C_1\delta-\frac{1}{t}(\delta r^2-4\delta^2r^2)\r]}.
\end{align*}
The proof of (b) is completed by taking $\delta:=1/8$.
\end{proof}

\begin{proposition}\label{2l2}
Let $x\in \partial M$ and  $\tilde{\sigma} _r$ be the same as in Proposition
\ref{2l1} for a fixed constant $r>0$. Then,
\begin{enumerate}
  \item [$(a)$] $\mathbb{E}^xe^{\lambda l_{t\wedge\tilde{\sigma}_r}}<\infty   \mbox{ for\ any}\ \lambda>0$;
  \item [$(b)$] $\mathbb{E}^xl_{t\wedge
  \tilde{\sigma}_r}=\frac{2\sqrt{t}}{\sqrt{\pi}}+O(t^{3/2})$ holds for small
  $t>0$.
\end{enumerate}
\end{proposition}
\begin{proof}
By using Proposition \ref{2l1}, and applying the estimate for $\rho_t^{\partial}$, the proof is similar to that of \cite[Theorem 2.1]{W09a} for constant manifolds, we omit it here.
\end{proof}

 \begin{proof}
 By using the It\^{o} formula, the equality \eqref{fKol} follows directly. Moreover, (ii) can be calculated by combining \eqref{fKol} and \eqref{eq-2.2}. Thus
it suffices for us to  prove (i).

Since $L_t$ is strictly elliptic and the coefficient is $C^{\infty}([0,T_c)\times M)$,  it is easy to know from Malliavin calculus that $P_{\cdot,t}f\in C^{1,2}([0,t]\times M)$.
To prove (i), it suffices to consider $x\in M^{\circ}:=M\setminus\partial
 M$ and $s=0$. Let $r_0>0$ be such that $B_0(x, r_0)\in
 M^{\circ}$, and take $h\in C^{\infty}(M)$ such that
 $$h|_{B_0(x,r_0/2)}=1,\ \mbox{and}\ h|_{B_0(x,r_0)^c}=0.$$
 By the It\^{o} formula, we have
 $$\vd (hP_{s,t}f)(X_s)=\vd M_s+\l\{L_s(hP_{s,t}f)+h\frac{\vd}{\vd s}P_{s,t}f\r\}(X_s)\vd s,$$
where $M_s$ is a martingale. Then,
 \begin{align}\label{2e1}\lim_{s\downarrow
 0}\frac{\mathbb{E}^x(hP_{s,t}f(X_s))-P_{0,t}f(x)}{s}&=\lim_{s\downarrow
 0}\mathbb{E}^x\frac{1}{s}\int_0^s\l\{L_r(hP_{r,t}f)+h\frac{\vd}{\vd
 r}P_{r,t}f\r\}(X_r)\vd r\nonumber\\
 &=\l\{L_r(P_{r,t}f)+\frac{\vd}{\vd r}P_{r,t}f\big|_{r=0}\r\}(x).
 \end{align}
 On the other hand, by Proposition \ref{2l1},
 $$\mathbb{E}^x(hP_{s,t}f)(X_s)-P_{0,t}f(x)=\mathbb{E}^{x}((h-1)P_{s,t}f)(X_s)\leq \|f(h-1)\|_{\infty}e^{-c/s},\ s\in (0,1]$$
holds for some constant $c>0$. Combining this with (\ref{2e1}), we
conclude that
$$\frac{\vd }{\vd r}\big|_{r=0}P_{r,t}f(x)=-L_0P_{0,t}f(x).$$

  Let $x\in M$. If $N_sP_{s,t}f\neq 0$,  $0\leq s<t<T_c$. For
instance $N_0P_{0,t}f(x)>0$, there exist $r_0>0$, small
$t>t_0>0$ and $\varepsilon>0$, such that $N_rP_{r,t}f(x)>\varepsilon$ holds on
$B_r(s,2r_0)$ and $r\in (0,t_0)$. Moreover, we assume $f\geq
0$. Let $h\in C^{\infty}([0,T_c)\times M)$, $0 \leq h\leq 1$  such that
$$N_rh_r=0,\ \ h_r|_{B_r(x,r_0)}=1\ \ \mbox{and}\ \ h_r|_{B_r(x,2r_0)^c}=0.$$
By the It\^{o} formula,
\begin{align}\label{2eq3}
\mathbb{E}^x\l(h_sP_{s,t}f(X_s)\r)=&P_{0,t}f(x)+\int_0^sP_{0,r}\l[L_r(h_rP_{r,t}f)+\frac{\vd
h_r}{\vd r}\cdot P_{r,t}f+h_r\frac{\vd}{\vd r}P_{r,t}f\r](x)\vd
r\nonumber\\
&+\mathbb{E}^x\int_0^sh_r N_rP_{r,t}f(X_r)\vd l_r.
\end{align}
Moreover, by Proposition \ref{2l1},
\begin{align*}\mathbb{E}^x(h_sP_{s,t}f)(X_s)-P_{0,t}f(x)=&\mathbb{E}^{x}((h-1)P_{s,t}f)(X_s)\leq \|f(h-1)\|_{\infty}\mathbb{P}^x(\tilde{\sigma}_{r_0}\leq
s)\\
\leq & \|f(h-1)\|_{\infty}e^{-c/s} ,\ s\in (0,1]
\end{align*}
and
\begin{align*}&\lim_{s\downarrow
0}\frac{1}{s}\int_0^sP_{0,r}\l[L_r(h_rP_{r,t}f)+\frac{\vd h_r}{\vd
r}\cdot P_{r,t}f+h_r\frac{\vd}{\vd r}P_{r,t}f\r](x)\vd
r\\
=&L_0P_{0,t}f+\frac{\vd }{\vd
r}\big|_{r=0}P_{r,t}f+\partial_r\big|_{r=0}h(r,x)P_{0,t}f=0.
\end{align*}
The last equality comes from (a) and  $h(r,x)=1$. Combining this with
(\ref{2eq3}) we arrive at
$$\varepsilon\lim_{s\rightarrow 0}\frac{1}{s}\mathbb{E}^xl_{s\wedge \tilde{\sigma}_{r_0}}= 0,$$
which is impossible according to Proposition \ref{2l2}. We then complete the proof.

%
 \end{proof}

\begin{remark}\label{rem-add-1}
 For a fixed $T\in (0,T_c)$,  from Theorem \ref{2t1}, we see that $P_{t,T}f$ is a solution to  the following heat equation with Neumann boundary condition,
\begin{align}\label{cauchy}
\begin{cases}
\partial_t u(\cdot,x)(t)=-L_tu(t,\cdot)(x),  & (t,x)\in [0,T]\times M,\\
u(T,x)=f(x),& x\in M,\\
N_tu(t,\cdot)(x)=0, &  (x,t)\in \partial M \times (0,T].
\end{cases}
\end{align}
Then let $(X_t^T)_{t\in [0,T]}$ be the reflecting
$L_{(T-t)}$-diffusion process with semigroup $\{\overline{P}_{s,t}\}_{0\leq s\leq
t\leq T}$.  It is obvious that $\overline{P}_{T-t,T}f,\ t\in [0,T]$ solves the Neumann problem
\begin{align}\label{Chy}
\begin{cases}
\partial_t u(\cdot,x)(t)=L_tu(t,\cdot)(x),& (t,x)\in [0,T]\times M,\\
u(0,x)=f(x), & x\in M,\\
N_tu(t,\cdot)(x)=0, & x\in \partial M,\  t\in (0,T].
\end{cases}
\end{align}
Actually, the theory, presented in this paper, is
meant to be applied to the solution of \eqref{Chy}.
\end{remark}

\subsection{Derivative formula and applications to characterizing $\mathcal{R}_t^{Z}$ and ${\rm II}_t$ }
\quad
This subsection is devoted to the derivative formula for the Neumann semigroup, which is further applied to characterizing $\mathcal{R}_t^{Z}$ and ${\rm II}_t$.

Before moving on, let us introduce some basic notations first.
For $u\in \mathcal{O}_t(M)$,  the lift operators $\mathcal{R}_t^{Z}(u),\ {\rm II}_t(u)\in \mathbb{R}^d\otimes\mathbb{R}^d$ are defined by
\begin{align*}
&\mathcal{R}^{Z}_t(u)(a,b)=\l<\mathcal{R}^{Z}_t(u)a,b\r>=\mathcal{R}_t^{Z}(ua,
u b),\ \
 {\rm II}_t(u)(a,b)={\rm II}_t({\bf p}^t_{\partial}ua,{\bf p}^t_{\partial}ub),\ \ a,b\in \mathbb{R}^d,
\end{align*}
where for $x\in \partial M$, ${\bf p}^t_{\partial}: T_{x}M\rightarrow T_x\partial M$ is the project operator on $(M,g_t)$.
We now introduce  the derivative formula for the  Neumann semigroup first.
\begin{theorem}\label{Bis}
Let $0\leq s<t<T_c$ and $x\in M$ be fixed.
Let $K\in C([0,T_c)\times M)$ and $\sigma\in C([0,T_c)\times \partial
M)$ be such that $\mathcal{R}_t^{Z}\geq
K_t$ and ${\rm II}_t\geq \sigma_t$. Assume that
\begin{align}\label{2qt}
\sup_{u\in [s,t]}\mathbb{E}\l(\exp\l\{-\int_s^uK(r, X_r)\vd
r-\int_s^u\sigma(r, X_r)\vd l_r\r\}\bigg | X_s=x\r)<\infty.
\end{align}
Then there exists a progressively measurable process
$\{Q_{s,r}\}_{s\leq r\leq t}$ on $\mathbb{R}^d\otimes
\mathbb{R}^d$ such that
$$Q_{s,s}=I,\ \ \|Q_{s,r}\|\leq \exp\l[-\int_s^rK(u, X_u)\vd
u-\int_s^r\sigma(u, X_u)\vd l_u\r],\ \ r\in [s,t].$$
Moreover,  for any
$f\in C_c^1(M)$ with $|\nabla^{\cdot}P_{{\cdot},t}f|_{\cdot}$ being bounded on
$[s,t]\times M$, and  $h\in C^1([s,t])$ satisfying $h(s)=0, h(t)=1$,
it holds
\begin{align}\label{2Bis}
(u_s)^{-1}\nabla^sP_{s,t}f(x)&=\mathbb{E}\l\{Q_{s,t}^*u_t^{-1}\nabla^tf(X_t)\big|
X_s=x\r\} \notag\\
&=\frac{1}{\sqrt{2}}\mathbb{E}\l\{f(X_t)\int_s^th'(r)Q^*_{s,r}\vd
B_r\bigg| X_s=x\r\}.
\end{align}
\end{theorem}
\begin{proof}
Without loss of generality, we assume $s=0$, and simply  denote
$Q_{0,t}$ by $Q_t$.

 Following the idea of \cite[Theorem 4.2]{Hsu}, we need to construct the multiplicative functional  $Q_s$ first. For any $n\geq 1$, let
$Q_s^{(n)}$ solve the equation
$$ \begin{cases}
\vd Q_s^{(n)}=-\mathcal{R}_s^{Z}(u_s)Q_s^{(n)}\vd
s-{\rm II}_s(u_s)Q_s^{(n)}\vd l_s\\
\qquad \qquad
-\frac{1}{2}(n+2\sigma(s,X_s)^+)\l[(Q_s^{(n)})^*u_s^{-1}N_s\r]\otimes\l(u_s^{-1}N_s\r)\vd
l_s,\\
Q_0^{(n)}=I.
\end{cases}$$
 It is easy to see that for any $a\in \mathbb{R}^d$,
\begin{align*}
 \vd \|Q_s^{(n)}a\|^2&=2\l<\vd Q^{(n)}_sa,
Q_s^{(n)}a\r>\\
&=-2\mathcal{R}_s^{Z}(u_sQ_s^{(n)}a, u_sQ_s^{(n)}a)\vd
s-2{\rm II}_s({\bf
p}^s_{\partial}u_sQ_s^{(n)}a,{\bf p}^s_{\partial}u_sQ_s^{(n)}a)\vd l_s\\
&\quad -[n+2\sigma(s,X_s)^+]\l<u_sQ_s^{(n)}a, N_s\r>_s^2\vd
l_s\\
&\leq -2\|Q_s^{(n)}a\|^2\l[K(s, X_s)\vd s+\sigma(s, X_s)\vd
l_s\r]-n\l<u_sQ_s^{(n)}a, N_s\r>_s^2\vd l_s,
\end{align*}
where $\|\cdot\|$ is the operator norm on $\mathbb{R}^d$ and $\l<\cdot,\cdot\r>$  denotes the inner product on $\mathbb{R}^d$.
Therefore,
\begin{align}\label{2eq1}\|Q_s^{(n)}\|^2\leq
\exp{\l[-2\int_0^sK(r, X_r)\vd r-2\int_0^s\sigma(r, X_r)\vd
l_r\r]}<\infty,\end{align}
and for any $m\geq 1$,
\begin{align}\label{2eq2}
&\lim_{n\rightarrow\infty}\mathbb{E}^x\int_0^{t\wedge\zeta_m}\|(Q_s^{(n)})^*u_s^{-1}N_s\|^2\vd
l_s \nonumber\\
&\quad \leq
\lim_{n\rightarrow\infty}\l(\frac{1}{n}+\frac{1}{n}\mathbb{E}^x\int_0^{t\wedge\zeta_m}2\|Q_s^{(n)}\|^2\l[|K|(s,X_s)\vd
s+|\sigma|(s, X_s)\vd l_s\r]\r)=0,
\end{align}
where the second equality follows from Proposition \ref{2l1},
(\ref{2eq1}) and the boundedness of $K$ and $\sigma$ on
$\{(s,y): \rho_s(x,y)\leq m,\ 0\leq s\leq t\}$. Combining this with (\ref{2eq1}) and (\ref{2qt}), we see that
$$\mathbb{E}^x\int_0^t\sup_{n\geq 1}\|Q_s^{(n)}\|\vd s+\mathbb{E}^x\sup_{n\geq 1}\|Q_t^{(n)}\|<\infty.$$
Thus, there exists a subsequence $\{Q^{(n_k)}\}$ and a progressively
measurable process $Q$ such that for any bounded measurable process
$(\varphi_s)_{s\in[0,t]}$ on $\mathbb{R}^d$ and any
$\mathbb{R}^d$-valued random variable $\eta$, it holds
$$\lim_{k\rightarrow\infty}\l\{\mathbb{E}^x\int_0^t(Q_s^{(n_k)}-Q_s)\varphi_s\vd s+\mathbb{E}^x(Q_t^{(n_k)}-Q_t)\eta\r\}=0.$$

 Next, we turn to prove the first equality in \eqref{2Bis}. By using the It\^{o} formula,  we have
\begin{align}\label{S2-add1}
\vd ({\bf d}P_{s,t}f)(X_s)=&\nabla^s_{u_s{\bf d} B_s}({\bf d}
P_{s,t}f)(X_s)+{\rm
Ric}_s^{Z}(\cdot,\ \nabla^sP_{s,t}f)(X_s)\vd s\nonumber\\
&+\nabla^s_{N_s}({\bf d}P_{s,t}f)(X_s)\vd l_s,
\end{align}
where ${\bf d}$ is the exterior differential \footnote{ For 0-form $f$, its exterior differential ${\bf d}f$ is defined by $${\bf d} f(X):=X(f)=\l<\nabla^tf, X\r>_t,\ \ \mbox{for}\ X\in TM.$$} and  $${\rm
Ric}_s^{Z}(X,Y):={\rm Ric}_s(X,Y)-\l<\nabla^s_XZ_s, Y\r>_s,\quad  X,Y\in TM.$$
Now for any $a\in \mathbb{R}^d$,
\begin{align*}
\vd u_sQ_s^{(n)}a=&-{\rm
Ric}_s^{Z}(u_sQ_s^{(n)}a, \cdot)\vd s-{\rm II}_s({\bf
p}_{\partial}^su_sQ_s^{(n)}a,\cdot)\vd
l_s\\
&-\frac{1}{2}(n+2\sigma(s,X_s))^+\l<N_s,u_sQ_s^{(n)}a\r>_s\l<N_s,\cdot\r>_s\vd
l_s.
\end{align*}
By this and \eqref{S2-add1}, we have
\begin{align}\label{add-3}
\vd \l<\nabla^sP_{s,t}f(X_s), u_sQ_s^{(n)}a\r>_s=&{\rm
Hess}^s_{P_{s,t}f}(u_sQ_s^{(n)}a,u_s\vd B_s)+{\rm
Hess}^s_{P_{s,t}f}(u_sQ_s^{(n)}a,N_s)\vd l_s\nonumber\\
&-{\rm II}_s({\bf
p}_{\partial}^su_sQ_s^{(n)}a,\nabla^sP_{s,t}f(X_s))\vd l_s.
\end{align}
Moreover, since for any $v\in T_y\partial M$, $y\in \partial M$, we have
\begin{align*}
0=v\l<N_s, \nabla^sP_{s,t}f\r>_s(y)=\l<\nabla^s_{v}
N_s, \nabla^sP_{s,t}f\r>_s(y)+{\rm
Hess}^s_{P_{s,t}f}(v,N_s),
\end{align*}
which implies
$${\rm Hess}^s_{P_{s,t}f}(v,N_s)={\rm II}_s(v,\nabla^sP_{s,t}f)(y).$$
Combining this with \eqref{add-3}, we arrive at
\begin{align}\label{add-4}
&\vd \l<\nabla^sP_{s,t}f(X_s), u_sQ_s^{(n)}a\r>_s\nonumber\\
&={\rm Hess}^s_{P_{s,t}f}(u_sQ_s^{(n)}a,u_s\vd B_s)+{\rm
Hess}^s_{P_{s,t}f}(N_s,N_s)\l<u_sQ_s^{(n)}a, N_s\r>_s\vd l_s.
\end{align}
It follows from \eqref{add-4}, (\ref{2eq2}) and the boundedness of
$|\nabla^{\cdot}P_{\cdot,t}f|_{\cdot}$ on $[0,t]\times M$ that
\begin{align*}
\l<\nabla^0P_{0,t}f, u_0a\r>_0&=\lim_{m\rightarrow
\infty}\lim_{k\rightarrow \infty}\mathbb{E}^x\l<\nabla^{t\wedge
\zeta_m}P_{t\wedge \zeta_m, t}f(X_{t\wedge\zeta_m}), u_{t\wedge
\zeta_m}Q_{t\wedge \zeta_m}^{(n_k)}a\r>_{t\wedge \zeta_m}\\
&=\lim_{m\rightarrow\infty}\lim_{k\rightarrow
\infty}\mathbb{E}^x\l\{{\bf 1}_{\{t\leq \zeta_m\}}\l<\nabla^tf(X_t),
u_tQ_t^{(n_k)}a\r>_t\r\}\\
&=\mathbb{E}^x\l<\nabla^tf(X_t),u_tQ_ta\r>_t.
\end{align*}
This implies the first equality.

Finally, it only leaves us to show the second equality.  Since by the It\^{o} formula, we obtain
$$\vd P_{s,t}f(X_s)=\sqrt{2}\l<\nabla^sf(X_s),u_s\vd B_s\r>_s.$$
Therefore, we have
$$f(X_t)=P_{0,t}f(x)+\sqrt{2}\int_0^t\l<\nabla^sP_{s,t}f(X_s),u_s\vd B_s\r>_s.$$
So, for any $a\in \mathbb{R}^d$ and $m\geq 1$, it follows from
(\ref{2eq1}), (\ref{2eq2}) and the boundedness of
$\{|\nabla^sP_{s,t}f|_{s}\}_{s\in [0,t]}$ that
\begin{align*}
&\frac{1}{\sqrt{2}}\mathbb{E}^x\l\{f(X_t)\int_0^th'(s)\l<Q_sa,\vd
B_s\r>\r\}=\mathbb{E}^x\l\{\int_0^th'(s)\l<u_sQ_s a,
\nabla^sP_{s,t}f\r>_s(X_s)\vd s\r\}\\
&=\lim_{k\rightarrow\infty }
\mathbb{E}^x\l\{\int_0^th'(s)\l<u_sQ_s^{(n_k)}a,
\nabla^sP_{s,t}f\r>_s(X_s)\vd s\r\}\\
&=\lim_{m\rightarrow \infty}\lim_{k\rightarrow\infty
}\int_0^th'(s)\mathbb{E}^x\l\{\l<u_{s\wedge \zeta_m}Q_{s\wedge
\zeta_m}^{(n_k)}a,\nabla ^{s\wedge \zeta_m}P_{s\wedge
\zeta_m,t}f\r>_{s\wedge \zeta_m}(X_{s\wedge
\zeta_m})\r\}\vd s\\
&=\int_0^th'(s)\l<u_0a, \nabla^0P_{0,t}f\r>_0(x)\vd
s\\
&=\l<\nabla^0 P_{0,t}f(x), u_0a\r>_0.
\end{align*}
We complete the proof.
\end{proof}

 By localizing the process on a fixed domain, we obtain the following  local version of the derivative formula directly.
\begin{corollary}\label{2Bismut}  Assume $\mathcal{R}_r^{Z}\geq K_r$ and
${\rm II}_r\geq \sigma_r$ for some $K \in C([0,T_c)\times M)$ and $\sigma\in C([0,T_c)\times \partial M)$. Let $0\leq s\leq t<T_c$, $ x\in M$ and $D$
be a compact domain of $M$ such that $x\in
D^{\circ}$, the inner set of $D$. Let $X_t$ be a reflecting $L_t$-diffusion process starting from $x$ at time $s$ and  $\tau_D=\inf\{t\in [s,T_c): X_t\in \partial D, X_s=x\}$. Then for all $0\leq s\leq r\leq t$,  there exists a progressively
measurable process $\{Q_{s,r}\}_{r\in [s,t]}$ on
$\mathbb{R}^d\otimes\mathbb{R}^d$ such that
$$Q_{s,s}=I,\  \|Q_{s,r}\|\leq \exp{\l[-\int_s^{r\wedge\tau_D}K(u, X_u)\vd u-\int_s^{r\wedge\tau_D}\sigma(u,X_u)\vd l_u\r]}.$$
In addition, for any $\mathbb{R}_{+}$-valued process $h$ satisfying
$h(s)=0$, $h(r)=1$ for $r>t\wedge \tau_D$ and
$$\mathbb{E}\l(\int_s^th'(r)^2\vd r\r)^{\alpha}<\infty$$ for some
$\alpha>1/2$, it holds
$$u_s^{-1}\nabla^sP_{s,t}f(x)=\frac{1}{\sqrt{2}}\mathbb{E}\l\{f(X_{t\wedge\tau_D})\int_s^th'(r)Q_{s,r}^*\vd B_r\bigg| X_s=x\r\},\ \ f\in \mathscr{B}_b(M).$$
\end{corollary}

 By using the derivative formula established above, we have the following formulae to characterize $\mathcal{R}_t^{Z}$ and ${\rm II}_t$, respectively. When the metric is fixed,  the formulae for ${\rm Ric}$  were established  in \cite{BE} and \cite[Propositions 2.1 and 2.6]{Bakry}, and the formulae for second fundamental form were  proved by F.-Y. Wang \cite{WSe}. There formulae are always applied to proving  that some functional inequalities  imply corresponding curvature conditions.
\begin{theorem}\label{4t1}
For each $s\in [0,T_c)$, let $x\in M^{\circ}$ $($the inner set of $M$$)$ and  $X\in T_xM$  with
$|X|_s=1$. Let $f\in C_0^{\infty}(M)$ such that
  $f=0$ around the boundary, ${\rm Hess}_f^s(x)=0$ and $\nabla^sf=X$. Set $f_n=f+n$ for $n\geq 1$. Then,\\
 $(i)$ for any
$p>0$,
\begin{align}
\mathcal{R}^Z_s(X,X)
&=\lim_{t\downarrow
s}\frac{P_{s,t}|\nabla^tf|^p_t(x)-|\nabla^sP_{s,t}f|^p_s(x)}{p(t-s)};\label{1-2}
\end{align}
 $(ii)$
 for any $p>1$,
 \begin{align*}
\mathcal{R}^Z_s(X,X)
&=\lim_{n\rightarrow\infty}\lim_{t\downarrow s
 }\frac{1}{t-s}\l\{\frac{p[P_{s,t}f_n^2-(P_{s,t}f_n^{\frac{2}{p}})^p]}{4(p-1)(t-s)}-|\nabla^sP_{s,t}f|_s^2\r\}(x)\nonumber\\
 &=\lim_{n\rightarrow\infty}\lim_{t\downarrow s}\frac{1}{t-s}\l\{P_{s,t}|\nabla^tf|^2_t-\frac{p[P_{s,t}f_n^2-(P_{s,t}f_n^{\frac{2}{p}})^p]}{4(p-1)(t-s)}\r\}(x);
\end{align*}
$(iii)$ $\mathcal{R}^Z_s(X,X)$
is equal to each of the following limits:
\begin{align*}
&\lim_{n\rightarrow\infty}\lim_{t\downarrow
s}\frac{1}{(t-s)^2}\l\{(P_{s,t}f_n)\l[P_{s,t}(f_n\log
f_n)-(P_{s,t}f_n)\log
P_{s,t}f_n\r]-(t-s)|\nabla^sP_{s,t}f|_s^2\r\}(x);\\
&\lim_{n\rightarrow\infty}\lim_{t\downarrow
s}\frac{1}{4(t-s)^2}\l\{4(t-s)P_{s,t}|\nabla^tf|^2_t+(P_{s,t}f_n^2)\log
P_{s,t}f_n^2-P_{s,t}{f_n^2\log f_n^2}\r\}(x).
\end{align*}
\end{theorem}
\begin{proof}
Without loss of generality, we only consider $s=0$. Let $r>0$ and $t_0\in (0,T_c)$ be such that $B_t(x,r)\subset M^{\circ},\ t\in [0,t_0]$ and
$|\nabla^tf|_t\geq \frac{1}{2}$ on $\{(t,x):t\in [0,t_0], x\in B_t(x,r)\subset M^{\circ}\}$. Due to Proposition \ref{2l1}, the proof of \cite[Theorem 4.1]{cheng}  works for the
present setting by replacing $s$ with $s\wedge \tilde{\sigma}_r$, where
$\tilde{\sigma}_r:=\inf\{s: X_s\notin B_s(x,r), \ X_0=x,\ s\in [0,t_0]\}$ and set $t_0=\inf\varnothing$ by convention, so that the
boundary condition needs not to be considered.

To avoid redundancy, we only prove (i) to explain the idea.  By Proposition \ref{2l1} and ${\rm Hess}_f^0(x)=0$, we have
\begin{align}\label{add-eq-8}
P_{0,t}|\nabla^tf|_t^p&=\mathbb{E}^x\{|\nabla^t f|_t^p(X_{t\wedge \tilde{\sigma}_r})\}+{\rm o}(t)\nonumber\\
&=|\nabla^0f|_0^p+\l[\frac{p}{2}|\nabla^0f|_0^{p-2}L_0|\nabla^0f|_0^2-\frac{p}{2}|\nabla^0f|_0^{p-2}\partial_tg_t|_{t=0}(\nabla^0f,\nabla^0f)\r]t+{\rm o}(t),
\end{align}
where the second equality comes from the following formula,
$$\partial_t|\nabla^t f|_t^2=-\partial_t g_t(\nabla^tf, \nabla^t f).$$
Moreover, since $f\in C_0^{\infty}(M)$ and $f=0$ around the boundary, by the Kolmogorov
equation,
$$\frac{\vd}{\vd t}|\nabla^0 P_{0,t}f|^p_0|_{t=0}=p|\nabla^0f|_0^{p-2}\l<\nabla^0L_0f,\nabla^0f\r>_0,$$
we have
\begin{align*}
|\nabla^0P_{0,t}f|_0^p=|\nabla^0f|_0^p+p|\nabla^0 f|_0^{p-2}\l<\nabla^0L_0f,\nabla^0 f\r>_0 t+{\rm o}(t).
\end{align*}
Combining  this with \eqref{add-eq-8}  yields \eqref{1-2} for $s=0$.
\end{proof}
\begin{theorem}\label{II-form}
For each $s\in [0,T_c)$, let $x\in \partial M$ and $X\in T_xM$ with $|X|_s=1$. Then for any constant $p>0$ and
 $f\in C_0^{\infty}(M)$ such that $\nabla^s f(x)=X$,  it holds
\begin{align}\label{add-eq-10}
{\rm II}_s(X,X)&=\lim_{t\downarrow s}\frac{\pi}{2p\sqrt{t-s}}\l\{P_{s,t}|\nabla^{t}f|_{t}^p-|\nabla^sf|_s^p\r\}(x)\nonumber\\
&=\lim_{t\downarrow
s}\frac{\pi}{2p\sqrt{t-s}}\l\{P_{s,t}|\nabla^{t}f|_{t}^p-|\nabla^s P_{s,t}f|_s^p\r\}(x).
\end{align}
If moreover $f>0$, then for any $p\in [1,2]$,
\begin{align*}
{\rm II}_s(X,X)&=-\lim_{t\downarrow
s}\frac{3}{8}\sqrt{\frac{\pi}{t-s}}\l\{|\nabla^sf|^2_s+\frac{p[(P_{s,t}f^{2/p})^p-P_{s,t}f^2]}{4(p-1)(t-s)}\r\}(x)\\
&=-\lim_{t\downarrow
s}\frac{3}{8}\sqrt{\frac{\pi}{t-s}}\l\{|\nabla^sP_{s,t}f|^2_s+\frac{p[(P_{s,t}f^{2/p})^p-P_{s,t}f^2]}{4(p-1)(t-s)}\r\}(x),
\end{align*}
where when $p=1$, we set $\frac{(P_{s,t}f^{2/p})^p-P_{s,t}f^2}{p-1}$ as the following limit
\begin{align*}
\lim_{p\downarrow
1}\frac{(P_{s,t}f^{2/p})^p-P_{s,t}f^2}{p-1}=(P_{s,t}f^2)\log P_{s,t}f^2- P_{s,t}(f^2\log f^2).
\end{align*}
\end{theorem}
\begin{proof}
  Due to  Propositions \ref{2l1} and  \ref{2l2}, the proof is
straightforward. For readers' convenience, we include the proof of the first equality in \eqref{add-eq-10}.

Let $r>0$ and $t_0\in (0,T_c)$ such that  $|\nabla^tf|_t\geq \frac{1}{2}$ holds on $\{(x,t): x\in B_t(x,r), t\in [0,t_0]\}$. Let
$\tilde{\sigma}_r:=\inf\{t\in [0,t_0]: X_t\notin B_t(x,r)\}$ and $t_0:=\inf \varnothing$. As $N_s|\nabla^s f|_s^2=2{\rm II}_s(\nabla^s f, \nabla^s f)$ holds on $\partial M$. So, by using the It\^{o} formula and  Propositions \ref{2l1} and  \ref{2l2},
\begin{align*}
P_{0,t}|\nabla^tf|_t^p(x)=&\mathbb{E}^x|\nabla^t f|_t^p(X_{t\wedge \sigma_r})+{\rm o}(t)\\
=&|\nabla^0f|_0^p(x)+\mathbb{E}^x\int_0^{t\wedge \sigma_r}(L_s+\partial_s)|\nabla^s f|_s^p(X_s)\vd s\\
&+p\{|\nabla^s f|_s^{p-2}{\rm II}_s(\nabla^sf,\nabla^s f)\}(X_s)\vd l_s+{\rm o}(t)\\
=&|\nabla^0 f|_0^p(x)+\frac{2p\sqrt{t}}{\sqrt{\pi}}{\rm II}_0(X,X)+{\rm o}(\sqrt{t})
\end{align*}
holds for small $t>0$. This proves the first equality in \eqref{add-eq-10}.

Note that the additional terms, derived from the time derivative of the metric, have the order ${\rm o}(t)$. Here, from the discussion above,   we find that it does not need to take care of these terms larger than order ${\rm o}(\sqrt{t})$. Thus, the calculation is similar to that in the fixed metric case. For the rest of the proof, we refer the reader to
\cite[Theorem 1.2]{WSe} and \cite[Theorem 3.2.4]{Wbook2} for details.

\end{proof}

\section{Proof of main results}
\hspace{0.5cm}
 In Subsection 3.1, we construct the coupling processes under convex flows by parallel displacement and reflection first,  then using coupling method, we give the proof of Theorem \ref{main-th-1}. In Subsection 3.2, we complete the proof of Theorem \ref{P1} by conformal change of the metrics and also the coupling method. In Section 3.3, we applied Theorems \ref{main-th-1} and \ref{P1} to the forward Ricci flow with umbilic boundary.

\subsection{Proof of Theorem \ref{main-th-1} (Convex flow)}
\hspace{0.5cm} We first
 introduce the coupling method for the reflecting $L_t$-diffusion processes.
 Let ${\rm Cut}_t(x)$ be the set of the $g_t$ cut-locus of $x$
on $M$. Then, the $g_t$ cut-locus ${\rm Cut}_t$ and the
space time cut-locus ${\rm Cut_{ST}}$ are respectively defined by
\begin{align*}
{\rm Cut}_t&=\{(x,y)\in M\times M\ |\ y\in {\rm Cut}_{t}(x)\};\\
{\rm Cut_{ST}}&=\{(t,x,y)\in [0,T_c)\times M\times M\ |\ (x,y)\in {\rm
Cut}_t\}.
\end{align*}
Set $D(M)=\{(x,x)|x\in M\}$.
For $(x,y)\notin {\rm Cut}_t$, let $\{J_i^t\}_{i=1}^{d-1}$
be Jacobi fields along the minimal geodesic $\gamma$ from $x$ to $y$
with respect to the metric $g_t$ such that at points $x$ and $y$, $\{J_i^t,
\dot{\gamma}: 1\leq i\leq d-1\}$ is an orthonormal basis. Let
\begin{align}
I_{Z}(t,x,y)=&\sum_{i=1}^{d-1}\int_{\gamma}\l(\l<\nabla^t_{\dot{\gamma}}J_i^t, \nabla^t_{\dot{\gamma}}J_i^t\r>_t-\l<R_t(J_i^t,\dot{\gamma})\dot{\gamma}, J_i^t\r>_t\r)(\gamma(s))\vd s+\frac{1}{2}\int_{\gamma}\partial_tg_t(\dot{\gamma}(s),\dot{\gamma}(s))\vd
s\nonumber\\
&+Z_t\rho_t(\cdot, y)(x)+Z_t\rho_t(x, \cdot)(y),\end{align}
where $R_t$ is the sectional curvature tensor with respect to the metric $g_t$.
Moreover, let $P_{x,y}^t: T_xM\rightarrow T_yM$ be the $g_t$-parallel
transform along the geodesic $\gamma$, and let
$$M^t_{x,y}: T_xM\rightarrow T_yM;\ v\mapsto P_{x,y}^tv-2\l<v,\dot{\gamma}\r>_t(x)\dot{\gamma}(y) $$
be the mirror reflection associated with the metric $g_t$. Then $P_{x,y}^t$ and $M_{x,y}^t$
are smooth outside ${\rm Cut}_t\cup D(M)$. For convenience,
 set $P_{x,x}^t$ and $M_{x,x}^t$ be the identity for $x\in
M$.
\begin{lemma}\label{coupling}
Let $x\neq y$ and $0<T<T_c$ be fixed. Let $U:[0,T)\times M\times
M\rightarrow TM$ be $C^1$-smooth in $({\rm
Cut_{ST}}\cup[0,T]\times D(M))^{c}$ such that $U(t,x_1,x_2)\in T_{x_2}M$ for $(t,x_1,x_2)\in [0,T]\times M\times
M$.
\begin{enumerate}
  \item[$(a)$] There exist two Brownian motions $B_t$ and $\tilde{B}_t$
 on the probability space $(\Omega, \{\mathscr{F}_t\}_{t\geq 0},
 \mathbb{P})$ such that
 $${\bf 1}_{\{(X_t, \tilde{X}_t)\notin{\rm Cut}_t\}}\vd \tilde{B}_t={\bf 1}_{\{(X_t, \tilde{X}_t)\notin{\rm Cut}_t\}}\tilde{u}_t^{-1}P^t_{X_t,\tilde{X}_t}u_t\vd B_t$$
holds, where $X_t$ with lift $u_t$ and local time $l_t$, and $\tilde{X}_t$ with lift
$\tilde{u}_t$ and and local time $\tilde{l}_t$ solve the equation
\begin{align}\label{3e4}
\begin{cases}
\vd X_t=\sqrt{2} u_t\circ\vd B_t+Z_t(X_t)\vd t+N_t(X_t)\vd l_t, &
X_0=x,\vspace{0.3cm}
\\
\vd \tilde{X}_t= \sqrt{2}\tilde{u}_t\circ \vd
\tilde{B}_t
+\l\{Z_t(\tilde{X}_t)+U(t,X_t,\tilde{X}_t){\bf 1}_{\{X_t\neq
\tilde{X}_t\}}\r\}\vd t+N_t(\tilde{X}_t)\vd \tilde{l}_t, &
\tilde{X}_0=y.
\end{cases}
\end{align}
Moreover,
for any $J\in C([0,T]\times M\times M)$ such that $J\geq I_Z$ on $({\rm Cut}_{\rm ST}\cup[0,T]\times D(M))^c$,
\begin{align}\label{3e5}\vd \rho_t(X_t,
\tilde{X}_t)&\leq
\bigg\{J(t,X_t,\tilde{X}_t)+\l<U(t,X_t,\tilde{X}_t),\nabla^t\rho_t(X_t,\cdot)(\tilde{X}_t)\r>_t{\bf
1}_{\{X_t\neq \tilde{X}_t\}}\bigg\}\vd t\end{align}
holds up to the coupling time $T_0:=\inf\{t\in [0,T]: X_t=\tilde{X}_t\}$, $\inf \varnothing =T$.
  \item[$(b)$] The first assertion in (a) holds with
  $M^t_{X_t,\tilde{X}_t}$ in place of
  $P_{X_t,\tilde{X}_t}^t$. In this case, for any $J\in C([0,T]\times M\times M)$ such that $J\geq I_Z$ on $({\rm Cut}_{\rm ST}\cup[0,T]\times D(M))^c$,
\begin{align}\label{3e6}
\vd \rho_t(X_t, \tilde{X}_t)\leq 2\sqrt{2}\vd
b_t+\bigg\{& J(t,X_t,\tilde{X}_t)+\l<U(t,X_t,\tilde{X}_t),\nabla^t\rho_t(X_t,\cdot)(\tilde{X}_t)\r>_t{\bf
1}_{\{X_t\neq \tilde{X}_t\}}\bigg\}\vd t \end{align}
holds up to the coupling time $T_0$, where $b_t$ is a
one-dimensional Brownian motion.
\end{enumerate}
\end{lemma}
\begin{proof}
We follow the argument in the proof of \cite[Theorem 3.4]{cheng}, but construct the coupling processes $(X_t,X_t^{n,\varepsilon})$ with reflecting boundary. Then we should add more argument for one more term caused by the local time on the boundary. More precisely,
when applying the It\^{o} formula to the radial process $\rho_t(X_t,X_t^{n,\varepsilon})$, we have the additional term
$${\bf 1}_{(M\times M)\setminus {\rm Cut}_t}(X_t,X_t^{n,\varepsilon})(N_t(X_t)+N_t(X_t^{n,\varepsilon}))\rho_t(X_t, X_t^{n,\varepsilon})\vd I_t^{n,\varepsilon},$$
where $I_t^{n,\varepsilon}$ is an increasing process which increasing only when $(X_t,Y_t^{n,\varepsilon})\in \partial (M\times M)\setminus {\rm Cut}_t$.
Thus to pass through the proof for the present case, we only need to show that  $N_t\rho_t(x,\cdot)(y)\leq 0$
for any $y\in \partial M, x\in M, (x,y)\in (M\times M)\setminus{\rm Cut}_t$ and $t\in [0,T_c)$. This is ensured by the convexity of the geometic flow. Therefore, the proof of \cite[Theorem 3.4]{cheng} also  works for the reflecting
$L_t$-diffusion case.
\end{proof}
By using the parallel coupling process, we complete the proof of Theorem \ref{main-th-1}.
\begin{proof}[Proof of Theorem \ref{main-th-1}]
First, by Theorem \ref{4t1}\,(i)(ii) and Theorem \ref{II-form} (i.e. the characterizations for $\mathcal{R}_t^Z$ and ${\rm II}_t$), each of \eqref{add-1}
and \eqref{add-2} implies \eqref{CV1} directly.

Now, suppose the curvature condition \eqref{CV1} holds. We prove \eqref{add-1}.  We first observe from the index lemma that
\begin{align}\label{add-5}
I_Z(t,x,y)\leq& \frac{1}{2}\int_0^{\rho_t(x,y)}\partial_tg_t(\dot{\gamma},\dot{\gamma})(\gamma(s))\vd s-\int_0^{\rho_t(x,y)}{\rm Ric}_t^Z(\dot{\gamma},\dot{\gamma})(\gamma(s))\vd s\nonumber\\
=&-\int_0^{\rho_t(x,y)}\mathcal{R}_t^Z(\dot{\gamma},\dot{\gamma})(\gamma(s))\vd s\notag\\
\leq &-K(t)\rho_t(x,y),
\end{align}
where $\gamma:[0,\rho_t(x,y)]\rightarrow M$ is the minimal geodesic from $x$ and $y$
associated with $g_t$. Now let $U=0$ and $(X_t,\tilde{X}_t)$ be the coupling by parallel displacement for $X_0=x, \tilde{X}_0=y$. By Lemma \ref{coupling} for $U=0$, \eqref{add-5} and $N_t\rho_t(x,\cdot)(y)\leq 0$ for $y\in \partial M$ and $x\in M$, we obtain
\begin{align*}
\vd \rho_t(X_t,\tilde{X}_t)
\leq&-K(t)\rho_t(X_t,\tilde{X}_t)\vd t.
\end{align*}
Thus, $\rho_t(X_t,\tilde{X}_t)\leq e^{-\int_s^tK(u)\vd u}\rho_s(X_s,\tilde{X}_s)$, which together with the dominated convergence theorem, we have
\begin{align*}
|\nabla^sP_{s,t}f(x)|_s\leq& \limsup_{y\rightarrow x}\frac{\mathbb{E}(|f(X_t)-f(\tilde{X}_t)|\,\big|\,(X_s,\tilde{X}_s)=(x,y))}{\rho_s(x,y)}\\
\leq & e^{-\int_s^tK(u)\vd u}\limsup_{y\rightarrow x}\mathbb{E}\l(\frac{|f(X_t)-f(\tilde{X}_t)|}{\rho_t(X_t,\tilde{X}_t)}\,\bigg|\,(X_s,\tilde{X}_s)=(x,y)\r)\\
\leq & e^{-\int_s^tK(u)\vd u} P_{s,t}|\nabla^tf|_t(x).
\end{align*}

Finally, we prove the Harnack inequality. Let $f\in C_0^{\infty}(M)$ be such that $f\geq 1$ and $f$ is constant outside a compact set. Given $x \neq y$ and $t>0$, let $\gamma:
[0,t]\rightarrow M$ be the $g_0$-geodesic from $x$ to $y$ with
length $\rho_0(x,y)$. Let $\nu_s=\frac{\vd \gamma_s}{\vd s}$. Then
we have $|\nu_s|_0=\rho_0(x,y)/t$. Let
$$h(s)=\frac{t\int_0^se^{2\int_0^rK(u)\vd u}\vd r}{\int_0^te^{2\int_0^rK(u)\vd u}\vd r}.$$
Then $h(0)=0$ and $h(t)=t$. Set $y_s=\gamma_{h(s)}$ and
$$\varphi(s)=\log P_{0,s}(P_{s,t}f)^p(y_s), \ s\in [0,t].$$
To get the derivative of $\varphi$, by using the It\^{o} formula,   we first have
\begin{align*}
{\rm d} (P_{s,t}f)^p(X_s)=\vd M_s +p(p-1)(P_{s,t}f)^{p-2}(X_s)|\nabla^sP_{s,t}f|_s^2(X_s)\vd s,\quad 0<s<\zeta_n,
\end{align*}
where $M_s$ is a local martingale. As explained above, $|\nabla^sP_{s,t}f|_s\leq e^{-\int_s^tK(r)\vd
r}P_{s,t}|\nabla^{t} f|_t$ and $(P_{s,t}f)^{p-2}$ is bounded, it is easy to deduce that
$$P_{0,s} (P_{s,t}f)^p(x)-(P_{0,t}f)^p(x)=p(p-1)\int_0^sP_{0,r}[(P_{r,t}f)^{p-2}|\nabla^rP_{r,t}f|_r^2](x)\vd r.$$
That is
$$\frac{\vd P_{0,s} (P_{s,t}f)^p(x)}{\vd s}=p(p-1)P_{0,s}[(P_{s,t}f)^{p-2}|\nabla^sP_{s,t}f|_s^2](x),$$
 which implies
that for any $s\in [0,t]$,
\begin{align*}
\frac{\vd \varphi(s)}{\vd
s}=&\frac{1}{P_{0,s}(P_{s,t}f)^p}\Big\{P_{0,s}\Big(p(p-1)(P_{s,t}f)^p|\nabla^s\log
P_{s,t}f|^2_s +h'(s)\l<\nabla^0P_{0,s}(P_{s,t}f)^{p},\nu_s\r>_0\Big\}\\
\geq&\frac{p}{P_{0,s}(P_{s,t}f)^p}P_{0,s}\Big\{(P_{s,t}f)^p\Big((p-1)|\nabla^s\log
P_{s,t}f|^2_s\\
&\hspace{4.5cm}-\frac{\rho_0(x,y)}{t}h'(s)e^{-\int_0^s K(u)\vd u}|\nabla^s\log
P_{s,t}f|_s\Big) \Big \} \\
\geq& \frac{-p\rho_0^2(x,y) h'(s)^2e^{-2\int_0^sK(u)\vd u}}{4(p-1)t^2}.
\end{align*}
Since $h'(s)=\frac{te^{\int_0^s2K(u)\vd
u}}{\int_0^te^{\int_0^r2K(u)\vd u}\vd r} $, we arrive at
$$\frac{\vd \varphi(s)}{\vd
s}\geq \frac{-p\rho_0(x,y)^2e^{\int_0^s2K(u)\vd u}
}{4(p-1)(\int_0^te^{2\int_0^rK(u)\vd u}\vd r)^2},\ \quad s\in [0,t].$$ By
integrating over $s$ from 0 and $t$, we complete the proof of \eqref{add-2} for $s=0$.
\end{proof}
\begin{remark}\label{rem-Harnack}
We point out that by letting  $\phi_t\equiv 1$ in the proof of  Theorem \ref{P1}\,(ii), the Harnack inequality can be deduced by using coupling method directly.
\end{remark}


\subsection{Proof of Theorem \ref{P1} (Non-convex flow)}
\hspace{0.5cm} The proof of Theorem \ref{P1} is divided into two parts. First, we
prove that the curvature condition {\bf (H1)} implies the gradient inequality.

\begin{proof}[{Proof of Theorem \ref{P1}}] (Gradient inequality)
We also  use coupling method to prove the gradient inequality.
To this end, we need make a conformal change of the geometric flow $g_t$ first. Let $\phi\in \mathscr{D}$.
As announced,  the new flow $\tilde{g}_t:=\phi_t^{-2}g_t$ is convex flow.
 Let $\tilde{\Delta}_t$ and $\tilde{\nabla}^t$ be the Laplacian and gradient operator associated with the metric $\tilde{g}_t$. According to \cite[(2.2)]{TW},
 \begin{align}\label{eq-add-3}
 L_t=\phi_t^{-2}(\tilde{\Delta}_t+\tilde{Z}_t)\  \ \mbox{and} \  \ \tilde{Z}_t=\phi_t^2Z_t+\frac{d-2}{2}\nabla^t\phi_t^2.
 \end{align}

To simplify the discussion, we consider the process generated by $L'_t=\varphi_t^2(\Delta_t+Z_t)$ on the manifold carrying convex flow $\{g_t\}_{t\in [0,T_c)}$ first, where $\varphi\in C^{1,\infty}([0,T_c)\times M)$ and $0<\varphi\leq 1$.
Moreover, suppose
$${\rm Ric}^Z_t\geq k_1(t)\quad\mbox{and}\quad \partial_tg_t\leq k_2(t)$$
for some functions $k_1,k_2\in C([0,T_c))$.
Let $X_t$ solve
\begin{align}\label{add-eq-2}
\vd_IX_t=\sqrt{2}\varphi_t(X_t)u_t\vd B_t+\varphi_t^2(X_t)Z_t(X_t)\vd t+N_t(X_t)\vd l_t,\quad X_0=x.
\end{align}
Let $Y_t$ solve
\begin{align}\label{add-eq-3}
\vd_IY_t=\sqrt{2}\varphi_t(Y_t)P_{X_t,Y_t}^tu_t\vd B_t+\varphi_t^2(X_t)Z_t(Y_t)\vd t+N_t(Y_t)\vd \tilde{l}_t,\quad Y_0=y.
\end{align}
As the boundary $(\partial M, g_t)$ is convex for all $t\in [0,T_c)$, by the It\^{o} formula, we have
\begin{align*}
\vd \rho_t(X_t,Y_t)\leq &\sqrt{2}(\varphi_t(X_t)-\varphi_t(Y_t))\vd b_t+
\bigg\{\sum_{i=1}^{n}(U_i^t)^2\rho_t(X_t,Y_t)+\partial_t\rho_t(X_t,Y_t)\\
&+\l<\varphi_t^2Z_t(Y_t),\nabla^t\rho_t(X_t,\cdot)(Y_t)\r>_t+\l<\varphi_t^2Z_t(X_t),\nabla^t\rho_t(\cdot, Y_t)(X_t)\r>_t\bigg\}\vd t,
\end{align*}
where $b_t$ is a one-dimensional Brownian motion, $\{U^t_i\}_{i=1}^{n}$ are vector fields on $M\times M$ such
that ${\nabla}^t U^t_i(X_t,Y_t)=0$ and
$$U^t_i(X_t,Y_t)=\varphi_t(X_t)V^t_i+\varphi_t(Y_t){P}^t_{X_t,Y_t}V^t_i,\ \ 1\leq i\leq n$$
for $\{V^t_i\}_{i=1}^{n}$ a ${g}_t$-orthonormal basis of $T_{X_t}M$. Let ${\rho}_t={\rho}_t(X_t,Y_t)$.
Define
$${J^t_i}(s)=\l(\frac{s}{{\rho}_t}\varphi_t(Y_t)+\frac{{\rho}_t-s}{{\rho}_t}\varphi_t(X_t)\r){P}^t_{\gamma(0),\gamma(s)}V^t_i,\ \ 1\leq i\leq n,$$
where ${J^t_i}(0)=\varphi_t(X_t)V^t_i$ and
${J^t_i}({\rho}_t)=\varphi_t(Y_t){P}^t_{X_t,Y_t}V^t_i$. Note that
${P}^t_{\gamma(0),\gamma(s)}V^t_i$ are parallel vector fields along
$\gamma_t$,
\begin{align}\label{add-eq-7}
&\sum_{i=1}^{d}(U_i^t)^2\rho_t(X_t,Y_t)\notag\\
&\leq \sum_{i=1}^{d}\int_0^{\rho_t}\l\{|\nabla^t_{\dot{\gamma}}J_i^t|_t^2-\l<R_t(\dot{\gamma},J_i^t)J_i^t, \dot{\gamma}\r>_t\r\}(\gamma(s))\vd s\nonumber\\
&\leq  d\|\nabla^t\varphi_t\|^2_{\infty}\rho_t-\frac{1}{\rho_t^2}\int_0^{\rho_t}\{s\varphi_t(Y_t)+(\rho_t-s)\varphi_t(X_t)\}^2{\rm Ric}_t(\dot{\gamma}(s),\dot{\gamma}(s))\vd s.
\end{align}
On the other hand,
\begin{align}\label{add-eq-4}
\varphi_t^2&(X_t)\l<Z_t(X_t), \nabla^t\rho_t(\cdot, Y_t)(X_t)\r>_t+\varphi_t^2(Y_t)\l<Z_t(Y_t), \nabla^t\rho_t(X_t,\cdot)(Y_t)\r>_t\nonumber\\
&=
\frac{1}{\rho_t^2}\int_0^{\rho_t}\frac{\vd }{\vd s}\{(s\varphi_t(Y_t)+(\rho_t-s)\varphi_t(X_t))^2\l<Z_t({\gamma(s)}),\dot{\gamma}(s)\r>_t\}\vd s\nonumber\\
&\leq \frac{1}{\rho_t^2}\int_0^{\rho_t}(s\varphi_t(Y_t)+(\rho_t-s)\varphi_t(X_t))^2\l<(\nabla^t_{\dot{\gamma}}Z_t)\circ \gamma, \dot{\gamma}\r>_t(\gamma(s))\vd s+2\|Z_t\|\|\nabla^t\varphi_t\|_{\infty}\rho_t.
\end{align}
Moreover,
$$\partial_t \rho_t(X_t,Y_t)=\frac{1}{2}\int_0^{\rho_t}\partial_tg_t(\dot{\gamma}(s),\dot{\gamma}(s))\vd s\leq \frac{1}{2}k_2(t)\rho_t.$$
Combining this with \eqref{add-eq-7} and \eqref{add-eq-4},  we have
\begin{align}\label{add-eq-5}
\vd \rho_t(X_t,Y_t)\leq &\sqrt{2}(\varphi_t(X_t)-\varphi_t(Y_t))\vd b_t- k_1(t)\l\{\frac{1}{\rho_t^2}\int_0^{\rho_t}[s\varphi_t(Y_t)+(\rho_t-s)\varphi_t(X_t)]^2\vd s\r\}\vd t\notag \\
&+
\l\{d\|\nabla^t\varphi_t\|_{\infty}^2\rho_t+2\|Z_t\|_{\infty}\|\nabla^t\varphi_t\|_{\infty}\rho_t+\frac{1}{2}k_2(t)\rho_t\r\}\vd t
&\nonumber\\
\leq & \sqrt{2}(\varphi_t(X_t)-\varphi_t(Y_t))\vd b_t+K_{\varphi}(t)\rho_t(X_t,Y_t)\vd t,
\end{align}
where
\begin{equation}\label{K-phi}
K_{\varphi}(t):=d\|\nabla^t\varphi_t\|_{\infty}^2+2\|Z_t\|_{\infty}\|\nabla ^t\varphi_t\|_{\infty}+k^{-}_1(t)+\frac{1}{2}k_2(t).
\end{equation}

Now we return to the diffusion processes generated by $L_t=\phi_t^{-2}(\tilde{\Delta}_t+\tilde{Z}_t)$ (see \eqref{eq-add-3}). Let $\varphi_t=\phi^{-1}_t$ and $\widetilde{{\rm Ric}}_t$ be the new Ricci curvature tensor with respect to the metric $\tilde{g}_t$.
By \cite[Theorem 1.129]{Bess} and \cite[(3.2)]{FWW}, for any $X\in TM$ such that ${\tilde{g}_t}(X,X)=1$, i.e. $|X|_t=\phi_t$, we have
$$\widetilde{{\rm Ric}}_t(X,X)={\rm Ric}_t(X,X)+(d-2)\phi_t^{-1}{\rm Hess}_{\phi_t}^t(X,X)+\frac{1}{2}\nabla^t\phi_t^2-(d-2)|\nabla^t\phi_t|^2_t,$$
and
\begin{align*}
\tilde{g}_t(\tilde{\nabla}^t_X\tilde{Z}_t,X)=&\l<\nabla^t_XZ_t,X\r>_t+2\l<\nabla^t\log\phi_t,X\r>_t\l<Z_t,X\r>_t\\
&+(d-2)\phi_t^{-1}{\rm Hess}^t_{\phi_t}(X,X)+(d-2)\l<X,\nabla^t\log\phi_t\r>^2_t\\
&-\phi_t\l<Z_t, \nabla^t\phi_t\r>_t-(d-2)|\nabla^t\phi_t|_t^2.
\end{align*}
Therefore, noting that $|X|_t=\phi_t$, we have
\begin{align*}
\widetilde{{\rm Ric}}_t^{\tilde{Z}}(X,X)&:=\widetilde{{\rm Ric}}_t(X,X)-\tilde{g}_t(\tilde{\nabla}^t_X\tilde{Z}_t,X)\\
&={\rm Ric}_t^Z(X,X)+\frac{1}{2}L_t\phi_t^2-2\l<\nabla^t\log\phi_t, X\r>_t\l<Z_t,X\r>_t -(d-2)\l<X,\nabla^t\log\phi_t\r>^2_t\\
&\geq K_1(t)\phi_t^2+\frac{1}{2}L_t\phi_t^2-|\nabla^t\phi_t^2|_t\cdot|Z_t|_t-(d-2)|\nabla^t\phi_t|^2_t\\
&\geq K_{\phi,1}(t),
\end{align*}
and
\begin{align*}
\partial_t\tilde{g}_t(X,X)&=\partial_t[\phi_t^{-2}g_t(X,X)]
=(\partial_t\phi_t^{-2})\phi_t^2+\phi_t^{-2}\partial_tg_t(X,X)\\
&\leq -2\partial _t\log\phi_t+K_{2}(t)\leq K_{\phi,2}(t).
\end{align*}
Moreover, let $|\cdot|_t'$ be the norm with respect to the metric $\tilde{g}_t$.
Then,
$$|\tilde{\nabla}^t\phi_t^{-1}|'_t\leq |\nabla^t \phi_t|_t\quad  \mbox{and} \quad   |\tilde{Z}_t|_t'\leq \l|\phi_tZ_t+\frac{d-2}{2}\nabla^t\phi_t^2\r |_t,$$
which, together with \eqref{add-eq-5}, yields
\begin{align*}
\vd \tilde{\rho}_{t}(X_t,Y_t)\leq \sqrt{2}(\phi_t^{-1}(X_t)-\phi_t^{-1}(Y_t)) \vd b_t+K_{\phi}(t)\tilde{\rho}_t(X_t,Y_t)\vd t,
\end{align*}
where $$K_{\phi}(t):=K_{\phi,1}^-(t)+\frac{1}{2}K_{\phi,2}(t)+2\|\phi_tZ_t+(d-2)\nabla^t\phi_t\|_{\infty}\|\nabla^t\phi_t\|_{\infty}+d\|\nabla^t\phi_t\|_{\infty}^2.$$
 In addition, $\phi_t\geq 1$, we therefore have $\tilde{\rho}_t\leq\rho_t\leq \|\phi_t\|_{\infty}\tilde{\rho}_t$, which implies
$$\mathbb{E}^{(x,y)}\rho_t(X_t,Y_t)\leq \|\phi_t\|_{\infty}e^{\int_s^tK_{\phi}(r)\vd r}{\tilde{\rho}_s}(x,y)\leq \|\phi_t\|_{\infty}e^{\int_s^tK_{\phi}(r)\vd r}{\rho}_s(x,y),\ s \leq t<T_c.$$
Then,
\begin{align*}
|\nabla^sP_{s,t}f|(x)=\lim_{y\rightarrow x}\bigg|\frac{P_{s,t}f(x)-P_{s,t}f(y)}{\rho_s(x,y)}\bigg|&=\bigg|\mathbb{E}^{(x,y)}\l[\frac{f(X_t)-f(Y_t)}{\rho_t(X_t,Y_t)}\frac{\rho_t(X_t,Y_t)}{\rho_s(x,y)}\r]\bigg|\\
&\leq \|\nabla^tf\|_{\infty}\|\phi_t\|_{\infty}e^{\int_s^tK_{\phi}(r)\vd r},
\end{align*}
which leads to complete the proof directly.
\end{proof}

The following result is derived from Theorem \ref{P1}\,(i) and Theorem \ref{Bis}.
 \begin{corollary}\label{c2}
Assume  ${\bf (H1)}$ holds.  If there exists $\phi\in \mathscr{D}$ such that $K_{\phi}(t)<\infty$ for all $ 0\leq t< T_c$, then
 for $p\in [1,\infty)$ and $f\in C^1(M)$ such that $f$ is constant outside a compact set,
 $$|\nabla^sP_{s,t}f|_s\leq \|\phi_t\|_{\infty}(P_{s,t}|\nabla^tf|_t^{p/(p-1)})^{(p-1)/p}e^{-\int_s^tK_{\phi}^{(p)}(r)\vd r},\ 0\leq s\leq t< T_c$$
 holds for $K_{\phi}^{(p)}(r):=\inf\{\phi_r^{-1}(L_r+\partial_r)\phi_r-(p+1)|\nabla^{r}\log\phi_r|^2_r\}.$
 Moreover, for $f\in \mathscr{B}_b(M)$,
\begin{align}\label{eq1}
|\nabla^sP_{s,t}f|^2_s\leq\frac{1}{2}\l[\int_s^t\|\phi_u\|_{\infty}^{-2}e^{2\int_s^uK_{\phi}^{(2)}(r)\vd r}\vd u\r]^{-1}P_{s,t}f^2,\ \ 0\leq s<t< T_c.
\end{align}
 \end{corollary}
\begin{proof}
The proof is due to \cite[Corollary 3.2.8]{Wbook2}.  We include it in Appendix for readers' convenience.
\end{proof}

We apply the coupling method to the proof of the Harnack inequality (See Theorem \ref{P1}\,(ii)). In \cite{WAP}, F.-Y.Wang constructed a proper coupling process to get the Harnack inequalities on manifolds with fixed metric. Here, we should modify the idea to our setting, where the main  difficulty is to construct
the coupling process such that it does not miss the information from the Ricci curvature.

\begin{proof}[Proof of Theorem \ref{P1}] (Part I: Harnack inequality)
Without loss of generality, we assume $s=0$ and $t=T$.
Now, let $x,y \in M$ and $T\in (0,T_c)$ be fixed.  To simplify the discussion, we also consider the process generated by ${L}'_t=\varphi_t^2(\Delta_t+Z_t)$ on a manifold carrying convex flow and suppose
$${\rm Ric}^Z_t\geq k_1(t)\quad \mbox{and}\quad \partial_tg_t\leq k_2(t)$$
for some $k_1,k_2\in C([0,T_c))$.

Let $X_t$ solve (\ref{add-eq-2}) with $X_0=x$. For some strictly
positive function $\xi_t\in C([0,T))$, let $Y_t$ solve
\begin{align}\label{SDE2}
\vd _I Y_t=\sqrt{2}\varphi_t(Y_t)P_{X_t, Y_t}^tu_t\vd B_t+\varphi^2_tZ_t(Y_t)\vd
t-\frac{\varphi_t(Y_t)\rho_t(X_t,Y_t)}{\varphi_t(X_t)\xi_t}\nabla^t\rho_t(X_t,\cdot)(Y_t)&\vd
t+N_t(Y_t)\vd \tilde{l}_t,\nonumber\\
&Y_0=y,
\end{align}
 where $\tilde{l}_t$ is the local time of $Y_t$ on
$\partial M$. In the spirit of Lemma \ref{coupling}, we may
assume that the cut-locus of $M$ is empty such that the parallel
displacement is smooth.
Let
\begin{align}\label{e8}
\vd \tilde{B}_t=\vd
B_t+\frac{\rho_t(X_t,Y_t)}{\sqrt{2}\xi_t\varphi_t(X_t)}u_t^{-1}\nabla^t\rho_t(\cdot,Y_t)(X_t)\vd
t,\ \ 0 \leq t<T.\end{align}
By a similar calculation as in \eqref{add-eq-5}, we have
\begin{align}\label{2eq6}
\vd \rho_t(X_t,Y_t)\leq &
\sqrt{2}(\varphi_t(X_t)-\varphi_t(Y_t))\l<\nabla^t\rho_t(\cdot,Y_t)(X_t),u_t\vd
\tilde{B}_t\r>_t+K_{\varphi}(t)\rho_t(X_t,Y_t)\vd
t\nonumber\\
&-\frac{\rho_t(X_t,Y_t)}{\xi_t}\vd t,\ \ \ 0\leq t<T,
\end{align}
which implies
\begin{align}\label{rho-1}
\vd\frac{\rho_t(X_t,Y_t)^2}{\xi_t}\leq &
\frac{2\sqrt{2}}{\xi_t}\rho_t(X_t,Y_t)(\varphi_t(X_t)-\varphi_t(Y_t))\l<\nabla^t\rho_t(\cdot,
Y_t)(X_t),u_t\vd
\tilde{B}_t\r>_t\nonumber\\
&-\frac{\rho_t(X_t,Y_t)^2}{\xi_t^2}\l[\xi_t'-(2\|\nabla^t\varphi_t\|^2_{\infty}+2K_{\varphi}(t))\xi_t+2\r]\vd
t,
\end{align}
where $K_{\varphi}$ is defined as in \eqref{K-phi}. Therefore,
by letting
$$\xi_t=(2-\theta)\int_t^Te^{-2\int_t^s({K}_{\varphi}(r)+\|\nabla^r\varphi_r\|_{\infty})\vd r}\vd s,\ \ t\in [0,T),\ \theta\in (0,2),$$
we know that $\xi_t, t\in [0,T)$ solves the following equation,
 $$2-(2\|\nabla^t\varphi_t\|^2_{\infty}+2K_{\varphi}(t))\xi_t+\xi'_t=\theta.$$
Combining this with \eqref{rho-1}, we obtain
\begin{align}\label{rho-estimate}
\vd\frac{\rho_t(X_t,Y_t)^2}{\xi_t}\leq &
\frac{2\sqrt{2}}{\xi_t}\rho_t(X_t,Y_t)(\varphi_t(X_t)-\varphi_t(Y_t))\l<\nabla^t\rho_t(\cdot,
Y_t)(X_t),u_t\vd
\tilde{B}_t\r>_t-\frac{\rho_t(X_t,Y_t)^2}{\xi_t^2}\theta\vd
t.
\end{align}
 Then, the following discussion is   similar  to that of \cite[Theorem 1.1]{WAP}, we omit it here.
\end{proof}

\subsection{Application to Ricci flow}
\hspace{0.5cm}
Now, we turn to consider the Ricci flow \eqref{Ricci-flow}. Assume that $\{g_t\}_{t\in [0,T]},\ T\in (0,\infty)$ is a complete solution to the equation \eqref{Ricci-flow}.  Let $\{P_{s,t}\}_{0\leq s\leq t\leq T}$  be the
Neumann diffusion semigroup generated by $\Delta_t$. Then, it is obvious to see that ${P_{s,T}f}$ is a solution to  the following heat equation
\begin{equation}\label{heat-equ-boundary}
\begin{cases}
  \frac{\partial}{\partial t}u(x,\cdot)(t)=-\Delta_tu(\cdot,t)(x),&\ \  (x,t)\in M\times [0,T];\\
  \\
  N_tu(\cdot,t)(x)=0,\ &\ \ x\in \partial M,\ t\in [0,T).
  \end{cases}
\end{equation}
    When $\lambda\geq 0$, the corresponding  gradient estimate and Harnack inequality can be derived from Theorem \ref{main-th-1} directly.
 \begin{theorem}\label{cor-2}
Suppose $\{g_{t}\}_{t\in [0,T]}$ is a complete solution to \eqref{Ricci-flow} with $\lambda\geq 0$. Then for $f\in C^1(M)$ such that $f$ is constant outside a compact set, $P_{s,T}f,\ s\in [0,T]$ is a solution to \eqref{heat-equ-boundary} and
\begin{equation}\label{Ricci-gd}
|\nabla ^sP_{s,T}f|_s\leq P_{s,T}|\nabla^Tf|_T,\ 0\leq s\leq T.
\end{equation}
Moreover, for $f\in \mathscr{B}_b(M)$ and $s\in [0,T]$,
\begin{equation}\label{Ricci-Hk}
(P_{s,T}f)^p(x)\leq P_{s,T}f^p(y)\exp{\l[\frac{p}{4(p-1)(T-s)}\rho_s^2(x,y)\r]}.
\end{equation}
\end{theorem}
\begin{remark}\label{convex-Ricci-flow-rem}
It is easy to see that these results above are similar to that for the Ricci flat manifold. Indeed, \eqref{Ricci-gd} and \eqref{Ricci-Hk} also can be derived when $\{g_t\}$ is a convex Ricci flow.
We would like to indicate that Pulemotov \cite{P10} gave the proof of the short time existence of the  convex Ricci flow.
\end{remark}

When $\lambda<0$, we need more curvature information around the boundary to deal with this case. Let ${\rm Sect}_t$ be the section curvature of $M$  and $\rho^{\partial}_t(x)$ be the distance between $x$ and $\partial M$ associated with the metric $g_t$. The required assumption is presented as follows.
\begin{description}
         \item[{\bf(H2) }] There exist positive constants $k,r_0,k_1$ such that $|{\rm Ric}_t|\leq k$  and  on the set $\partial^t_{r_0}M:=\{x\in M: \rho^{\partial}_t(x)\leq r_0\}$, $\rho^{\partial}_t$ is smooth and   ${\rm Sect}_t\leq k_1$.\vspace{-0.2cm}
       \end{description}
       If this assumption holds for $r_0<\frac{\pi}{2\sqrt{k_1}}$, then by constructing explicit $\phi_t$,  the constants in terms of $\phi$ in Theorem \ref{P1}  can be estimated. Thus, we have the gradient estimates for the solution to \eqref{heat-equ-boundary} by using  Theorem \ref{P1}\,(i) as follows.

\begin{theorem}\label{Ricci-gradient}
Suppose that $\{g_t\}_{t\in [0,T]}$ is a complete solution to \eqref{Ricci-flow} with $\lambda<0$.  Assume that the assumption
${\bf (H2)}$ holds for $0<r_0\leq \frac{\pi}{2\sqrt{k_1}}$. Then, for $f\in C^1(M)$ such that $f$ is constant outside a compact set, $P_{s,T}f$ is a solution to \eqref{heat-equ-boundary} and
\begin{align}
&|\nabla^sP_{s,T}f|_s\leq \l(1-\frac{\lambda r_0 d}{2}\r)\|\nabla^Tf\|_{\infty}\exp\l\{{(T-s)\l[-\frac{\lambda d}{r_0}+\l(4d-\frac{11}{2}\r)\lambda^2d^2-\lambda d r_0 k+2k\r]}\r\}.
\label{gradient-0-1}
\end{align}
\end{theorem}
\begin{proof}
From  the assumption {\bf (H2)}, we deduce that ${\rm Ric}_t\leq -k$ and $\partial_tg_t=2{\rm Ric}_t\leq 2k$, which leads to the following estimate
\begin{align*}
K_{\phi}(t)&\leq \inf\{\phi_t\Delta_t \phi_t\}^-+\inf\{\partial_t\log \phi_t\}^-+2k+(4d-6)\|\nabla^t\phi_t\|_{\infty}^2<\infty.
\end{align*}
We now turn  to  construct  an explicit  $\phi\in C^{1,2}([0,T]\times M)$.
Let
\begin{equation}\label{h}
  h(s)=\cos(\sqrt{k_1}\, s),\ \ \mbox{for all} \ \ s\geq 0.
\end{equation}
Then $0\leq h(s)\leq 1$ for  $s\in [0, \frac{\pi}{2\sqrt{k_1}}]$. Moreover, let
\begin{equation}\label{delta}
   \delta=\delta(r_0,\lambda,k_1)=\frac{-\lambda(1-h(r_0))^{d-1}}{\int_0^{r_0}(h(s)-h(r_0))^{d-1}\vd s}.
\end{equation}
Consider $$\phi_t:=\varphi\circ \rho_t^{\partial}, \quad \mbox{for all} \ \ t\in [0,T],$$ where
$$\varphi(r)=1+\delta\int_{0}^{r}(h(s)-h(r_0))^{1-d}\vd s\int_{s\wedge r_0}^{r_0}(h(u)-h(r_0))^{d-1}\vd u.$$
By an approximation argument, we may regard $\phi$ as $C^{\infty}([0,T]\times M)$-smooth.  Obviously, $\phi\geq 1$ and
$N_s\log \phi_s=-\lambda=-{\rm II}_s,\ \mbox{for all} \  s\in [0,T]$. So, $\phi \in \mathscr{D}$.

Next, we need to estimate $\inf\{\phi_t\Delta_t\phi_t\}^-$, $\inf\{\partial_t\log\phi_t\}^-$, $\|\nabla^t\phi_t\|_{\infty}^2$ and $\|\phi_t\|_{\infty}$ in terms of $\lambda, d, r_0$ and $k$.
As $h$ is decreasing on $[0, r_0]$, we conclude that
\begin{align*}
|\partial_t\log \phi_t|&=\l|\frac{\delta(h(\rho_t^{\partial}\wedge r_0)-h(r_0))^{1-d}\int_{\rho_t^{\partial}\wedge r_0}^{r_0}(h(u)-h(r_0))^{d-1}\vd u}{\phi_t}\partial_t \rho_t^{\partial}\r| \leq \frac{\delta r_0}{\phi_t}|\partial_t\rho_t^{\partial}|,\ \ \  \rho_t^{\partial}\leq r_0.
\end{align*}
Moreover, using the following formula,
$$\partial_t \rho_t^{\partial}=\frac{1}{2}\int_0^{\rho_t^{\partial}}\partial_tg_t(\dot{\gamma}(s),\dot{\gamma}(s))\vd s,\ \ \rho_t^{\partial}\leq r_0, $$
where $\gamma$ is the minimal geodesic from $x$ to $\partial M$, we obtain
 \begin{align}\label{phi-1}
 |\partial_t\log \phi_t|\leq \delta r_0^2k,\ \  \rho_t^{\partial}\leq r_0.\end{align}
Similarly, it holds
\begin{equation}\label{add-eq-phi}
|\nabla^t\phi_t|_t^2\leq \delta^2 r_0^2.
\end{equation}
In addition,
\begin{align*}
\int_{0}^{r_0}(h(s)-h(r_0))^{1-d}\vd s\int_{s}^{r_0}(h(u)-h(r_0))^{d-1}\vd u\leq \int_0^{r_0}(r_0-s)\vd s=\frac{r_0^2}{2},
\end{align*}
which implies
\begin{align}\label{eq6}
\|\phi_t\|_{\infty}=1+\delta \int_{0}^{r_0}(h(s)-h(r_0))^{1-d}\vd s\int_{s}^{r_0}(h(u)-h(r_0))^{d-1}\vd u\leq 1+\frac{\delta r^2_0}{2}.
\end{align}
Moreover,  since ${\rm II}_t=\lambda\leq 0$ and ${\rm Sect}_t\leq k_1$ on $\partial^t_{r_0}(M)$,  according to the Laplacian comparison theorem for $\rho^{\partial}_s$ (see \cite{Kasue1, Kasue2}),
we have $$\Delta_t\phi_{t}\geq \l(\frac{(d-1)\varphi'h'}{h}+\varphi''\r)(\rho_t^{\partial})\geq -\delta, \ 0<\rho_t^{\partial}\leq r_0 \l(\leq \frac{\sqrt{k_1}\pi}{2}\r),\  t\in [0,T].$$
Combining this with \eqref{eq6} implies
\begin{equation}\label{eq-add-1}
 \inf\{\phi_t\Delta_t\phi_t\}\geq -\|\phi_t\|_{\infty}\|\Delta_t\phi_t\|_{\infty}\geq -\l(1+\frac{\delta r_0^2}{2}\r)\delta^2r_0^2.
\end{equation}
 Concluding from  \eqref{phi-1}, \eqref{add-eq-phi}, \eqref{eq6} and \eqref{eq-add-1}, it suffices for us to estimate $\delta$. Since $-h'$ is increasing and $h$ is decreasing, by the FKG inequality, we have
$$\int_0^{r_0}(h(s)-h(r_0))^{d-1}\vd s\geq \frac{-r_0\int_0^{r_0}(h(s)-h(r_0))^{d-1}h'(s)\vd s}{-\int_0^{r_0}h'(s)\vd s}=\frac{r_0}{d}(1-h(r_0))^{d-1}.$$
By this and \eqref{delta}, we obtain $\delta\leq -\lambda d/r_0$.
Concluding all these estimates above, we have
\begin{align}
&K_{\phi,1}(t)^-\leq k-\frac{\lambda d}{r_0}+\l(d-\frac{3}{2}\r)\lambda^2d^2;\ \  K_{\phi,2}(t)\leq -2\lambda d k r_0+2k.\label{e-2}
\end{align}
Then,
\begin{align}
K_{\phi}(t)&\leq 2k-\frac{\lambda d}{r_0}+\l(4d-\frac{11}{2}\r)\lambda^2d^2-\lambda d r_0 k, 
\end{align}
which leads to complete the proof.
\end{proof}
In addition,
\begin{align}
K_{\phi}^{(p)}(r)&=\inf\{\phi_r^{-1}(\Delta_r\phi_r)+\partial_r \log \phi_r-(p+1)|\nabla^r\log \phi_r|_r^2\}\nonumber\\
&\geq  -\delta-\delta r_0k-\delta^2 r_0^2 (p+1)\nonumber\\
&\geq \frac{\lambda d}{r_0}+\lambda d r_0 k-\lambda^2r_0^2 (p+1).\label{esti-K-3}
\end{align}
Using this estimate and Corollary \ref{c2}, we have the following result directly.
\begin{corollary}\label{eq-2.3}
Under the some conditions of  Theorem \ref{Ricci-gradient},
we have that for $f\in C^1(M)$ such that $f$ is constant outside a compact set and $0\leq s<T$,
\begin{align*}
|\nabla^sP_{s,T}f|_s\leq \l(1-\frac{\lambda r_0 d}{2}\r)e^{(T-s)K_p}(P_{s,T}|\nabla^T f|_T^{p/(p-1)})^{(p-1)/p},
\end{align*}
where
\begin{align}\label{Kp}
K_p=-\frac{\lambda d}{r_0}-\lambda d r_0 k+\lambda^2r_0^2 (p+1).
\end{align}
Moreover, we have
\begin{align*}
|\nabla^sP_{s,T}f|_s^2\leq \l(1-\frac{\lambda r_0 d}{2}\r)^2\frac{K_2}{1-e^{2K_2(s-T)}} P_{s,T}f^2,\ \ s\in [0,T),\ f\in \mathscr{B}_b(M).
\end{align*}
\end{corollary}
Now, we  turn to consider  system  \eqref{Ricci-flow}-\eqref{heat-equ-boundary}. By Theorem \ref{P1}, we have
\begin{theorem}\label{Ricci-Harnack-inequality}
Under same condition of Theorem \ref{Ricci-gradient},
for $p>(\frac{2-2\lambda r_0 d}{2-\lambda r_0 d})^2$ and $0\leq s<T$, the Harnack inequality
\begin{align*}
(P_{s,T}f(y))^p\leq P_{s,T}f^p(x)\exp \l\{\frac{\sqrt{p}(\sqrt{p}-1)\tilde{K}\rho_s(x,y)}{8\tilde{\delta_p}[2(\sqrt{p}-1)/(2-\lambda r_0d)-\tilde{\delta_p}](1-e^{2\tilde{K}(s-T)})}\r\}
\end{align*}
holds for $\tilde{\delta_p}:=\max\l\{\frac{-\lambda r_0 d}{2-\lambda r_0 d}, \frac{\sqrt{p}-1}{2-\lambda r_0 d}\r\}$ and
$$\tilde{K}:=-\frac{\lambda d}{r_0}+\l(4d-\frac{9}{2}\r)\lambda^2d^2+2k-\lambda dkr_0.$$
\end{theorem}
\begin{proof}
It is easy to see from \eqref{eq6} that
\begin{align*}
\delta_T= \sup_{t\in [0,T]}(\sup\phi_t^{-1}-\inf \phi_t^{-1})\leq \frac{-\lambda r_0 d}{2-\lambda r_0 d};\ \
\lambda_T=\inf_{[0,T]\times M}\phi^{-1}\geq \frac{2}{2-\lambda r_0 d}.
\end{align*}
Combining this with the estimates obtained in the proof of Theorem \ref{P1}, we complete the proof.

\end{proof}

 \section{Equivalent functional  inequalities for curvature conditions }
\hspace{0.5cm} In this section, we  present the gradient estimates for the curvature conditions \eqref{CV1-1}, which is an extension of
\cite[Theorem 4.3]{cheng}  for the  time-inhomogeneous manifold without boundary. This part
is mainly based on \cite[Theorem 1.1]{WSe} for the case when the metric is independent of time.
\begin{theorem}\label{4t2}
Let   $p\in [1,\infty)$ and $\tilde{p}=p\wedge 2$. Then for
any $[s,t]\subset [0,T_c)$, $K\in C_b([s,t]\times M)$ and $\sigma\in C_b([s,t]\times
\partial M)$, the following statements are equivalent to each other.
\begin{enumerate}
  \item[$(i)$] $\mathcal{R}^Z_t\geq K_t$ and  $ {\rm II}_t\geq \sigma_t$ hold
for any $0\leq t<T_c$.
  \item[$(ii)$] $|\nabla^sP_{s, t}f (x)|_s^p\leq
\mathbb{E}\{|\nabla^tf|_t^p(X_t)\exp[-p\int_s^tK(r,X_{r})\vd
 r-p\int_s^t\sigma(r,X_r)\vd
l_r]|X_{s}=x\}$ holds for $x\in M$,  $0\leq s\leq t< T_c$, and $f\in
 C^1(M)$ such that $f$ is constant outside a compact set.
  \item [$(iii)$] For any $0\leq s\leq t< T_c$, $x\in M$ and positive $f\in C^1(M)$ such that $f$ is constant outside a compact set,
$$\frac{\tilde{p}[P_{s,t}f^2-(P_{s,t}f^{1/\tilde{p}})^{\tilde{p}}]}{4(\tilde{p}-1)}\leq \mathbb{E}\l\{|\nabla^tf|_t^2(X_t)\int_s^te^{-2\int_u^tK(r,X_{r})\vd r-2\int_u^t\sigma(r,X_r)\vd l_r}\vd u\bigg|X_s=x\r\},$$
where when $p=1$, the inequality is understood as its limit as
$p\downarrow 1$:
\begin{align*}
&P_{s,t}(f^2\log
f^2)(x)-(P_{s,t}f^2\log P_{s,t}f^2)(x) \\
\leq &\  4
\mathbb{E}\l\{|\nabla^tf|_t(X_t)\int_s^te^{-2\int_u^tK(r,X_{r})\vd r-2\int_u^t\sigma(r,X_r)\vd
l_r}\vd u\bigg|X_s=x\r\}.
\end{align*}
  \item[$(iv)$] For any $0\leq s< t< T_c$, $x\in M$ and positive function $f\in C^1(M)$ such that $f$ is constant outside a compact set,
 \begin{align*}&|\nabla^sP_{s,t}f|^2_s(x)\\
 \leq & \ \frac{[P_{s,t}f^{\tilde{p}}-(P_{s,t}f)^{\tilde{p}}](x)}{\tilde{p}(\tilde{p}-1)\int_s^t\l(\mathbb{E}\{(P_{u,t}f)^{2-\tilde{p}}(X_{u})e^{-2\int_s^uK(r, X_{r})\vd r-2\int_s^u\sigma(r,X_r)\vd l_r}|X_s=x\}\r)^{-1}\vd
 u},\end{align*}
where when $p=1$, the inequality is understood as its limit as
$p\downarrow 1$:
$$|\nabla^sP_{s,t}f|^2_s(x)\leq \frac{[P_{s,t}(f\log f)-(P_{s,t}f)\log P_{s,t}f](x)}{\dint_s^t\l(\mathbb{E}\l\{P_{u,t}f(X_{u})e^{-2\int_s^uK(r, X_{r})\vd r-2\int_s^u\sigma(r,X_r)\vd l_r}\big|X_s=x\r\}\r)^{-1}\vd u}.$$
\end{enumerate}
\end{theorem}
\begin{proof}
By the derivative formula established in Theorem \ref{Bis}, it is easy to derive (ii) from (i); then
according to Theorem \ref{4t1}, we see that (ii)--(iv) implies (i); and finally,
taking $f\in C^{\infty}(M)$ and $f$ is constant
outside a compact set, we derive (iii), (iv)
from (ii) by a  similarly discussion as in  the proof of \cite[Theorem 2.3.1]{Wbook2} for the case with constant metric. We just take the proof of ``(ii) $\Rightarrow$ (iii)" for example. A similar argument  leads to ``(ii)$\Rightarrow$ (iv)".

We again assume $s=0$.  As the boundedness of $|\nabla ^{\cdot}P_{\cdot,t}f|_{\cdot}$ on $[0, t]\times M$ is verified above,
by using the derivative formula in Theorem \ref{Bis},
\begin{align*}
&\frac{\vd}{\vd u}P_{0,u}(P_{u,t}f^{2/p})^p(x)\\
&=p(p-1)P_{0,u}\{(P_{u,t}f^{2/p})^{p-2}|\nabla^uP_{u,t}f^{2/p}|_u^2\}\\
&\leq \frac{4(p-1)}{p}\mathbb{E}^x\Big\{(P_{u,t}f^{2/p})^{p-2}(X_u)(P_{u,t}f^{\frac{2(2-p)}{p}})(X_u)\\
&\hspace{2cm} \times \mathbb{E}\l(|\nabla^tf|_t^2(X_t)e^{-2\int_u^tK(r,X_r)\vd r-2\int_u^t\sigma(r, X_r)\vd l_r}\big| \mathscr{F}_u\r)\Big\}
\end{align*}
holds for $x\in M$, $0\leq s\leq t<T_c$ and $f\in C^1(M)$ such that $f$ is constant outside a compact set.
Since $2-p\in [0,1]$, by the Jensen inequality and the Markov property, we arrive at
$$\frac{\vd }{\vd u}P_{0,u}(P_{u,t}f^{2/p})^p(x)\leq \frac{4(p-1)}{p}\mathbb{E}^x\l\{|\nabla^tf|_t^2(X_t)e^{-2\int_u^tK(r, X_r)\vd r-2\int_u^t\sigma(r, X_r)\vd l_r}\r\}.$$
Integrating with respect to $u$ over $[0,t]$ yields (iii) for $s=0$.
\end{proof}

Let $\varphi:\mathbb{R}^+\rightarrow \mathbb{R}^+$ be a non-decreasing function, we define
a cost function
$C_t(x,y)=\varphi(\rho_t(x,y)).$
To this cost function, we associate the Monge-Kantorovich minimization between
two probability measures on $M$,
\begin{align}\label{eq3}W_{C_t}(\mu,\nu)=\inf_{\eta\in \mathscr{C}(\mu,\nu)}\int_{M\times M}C_t(x,y)\vd \eta(x,y),\end{align}
where $\mathscr{C}(\mu,\nu)$ is the set of all probability measures on $M\times M$ with marginal $\mu, \nu\in \mathscr{P}(M)$ and $ \mathscr{P}(M)$ being the space of all probability measure on $M$. We denote
$$W_{p,t}(\mu,\nu)=(W_{\rho_t^p}(\mu, \nu))^{1/p}$$
the Wasserstein distance associated to $p>0$.

 We now list more equivalent statements for \eqref{CV1}, which is an extension
of \cite[Theorem 4.3]{cheng} to manifolds with boundary carrying convex flows.
 See also \cite{W09a, W10, WPC} for the corresponding conclusions on  the constant manifold with boundary.
\begin{theorem}\label{4t3}
Let $p\in [1,\infty)$, $K\in C([0,T_c))$ and  $\{p_{s,t}\}_{0\leq s\leq t<T_c}$ be
the heat kernel of $\{P_{s,t}\}_{0\leq s\leq t< T_c}$ associated with the volume measure $\mu_t$
  with respect to the metric $g_t$. Then the
following assertions are equivalent to each other.
\begin{enumerate}
  \item [$(i)$] \eqref{CV1} holds.
   \item [$(ii)$] For any $x,y\in M$ and $0\leq s< t< T_c$, $$W_{p,t}(\delta_xP_{s,t},\delta_yP_{s,t})\leq \rho_s(x,y)e^{-\int_s^tK(r)\vd
   r}.$$
  \item [$(ii')$] For any $\mu_1, \mu_2\in \mathscr{P}(M)$ and $0\leq s< t< T_c$,
  $$W_{p,t}(\nu_1P_{s,t},\nu_2P_{s,t})\leq W_{p,s}(\nu_1,\nu_2)e^{-\int_s^tK(r)\vd r}.$$
  \item [$(iii)$] When $p>1$, for any $f\in \mathscr{B}_b^+(M)$ and $0\leq s<t<T_c$,
  $$(P_{s,t}f)^p(x)\leq P_{s,t}f^p(y)\exp{\l[\frac{p}{4(p-1)}{\l(\int_s^te^{2\int_s^rK(u)\vd u}\vd
r\r)}^{-1}\rho_s^2(x,y)\r]}.$$
  \item [$(iv)$]For any $f\in \mathscr{B}_b^+(M)$ with $f\geq 1$ and $0\leq s\leq t< T_c$,
  $$P_{s,t}\log f(x)\leq \log P_{s,t}f(y)+{\l(4\int_s^te^{2\int_s^rK(u)\vd u}\vd
r\r)}^{-1}\rho_s^2(x,y).$$
  \item [$(v)$] When $p>1$, for any $0\leq s\leq t<T_c$ and $x,y \in M$,
 \begin{align*}
 &\int_Mp_{s,t}(x,y)\l(\frac{p_{s,t}(x,y)}{p_{s,t}(y,z)}\r)^{\frac{1}{p-1}}\mu_t(\vd z)
 \leq \exp{\l[\frac{p}{4(p-1)^2}{\l(\int_s^te^{2\int_s^rK(u)\vd u}\vd
r\r)}^{-1}\rho_s^2(x,y)\r]}.\end{align*}
 \item [$(vi)$] For any $0\leq s<t< T_c$ and $x,y\in M$,
 $$\int_M p_{s,t}(x,y)\log
 \frac{p_{s,t}(x,y)}{p_{s,t}(y,z)}\mu_{t}(\vd z)\leq
 \rho_s^2(x,y){\l(4\int_s^te^{2\int_s^rK(u)\vd u}\vd
r\r)}^{-1}.$$
\item [$(vii)$] For any $0\leq s<u\leq t<T_c$ and $1<q_1\leq q_2$ such
that
\begin{align}\label{3e11}
\frac{q_2-1}{q_1-1}=\frac{\int_s^te^{2\int_s^{\tau}K(r)\vd r}\vd
\tau}{\int_s^ue^{2\int_s^{\tau}K(r)\vd r}\vd \tau},
\end{align}
it holds
$$\{P_{s,u}(P_{u,t}f)^{q_2}\}^{\frac{1}{q_2}}\leq (P_{s,t}f^{q_1})^{\frac{1}{q_1}}, \ f\in\mathscr{B}^+_b(M).$$
\item [$(viii)$]  For any $0\leq s\leq u\leq t<T_c$ and $0<q_2\leq q_1$ or $q_2\leq
q_1<0$ such that $(\ref{3e11})$ holds,
$$(P_{s,t}f^{q_1})^{\frac{1}{q_1}}\leq \{P_{s,u}(P_{u,t}f)^{q_2}\}^{\frac{1}{q_2}},\ \ f\in \mathscr{B}^+_b(M).$$
\item [$(ix)$] For any $0\leq s<u\leq t<T_c$ and $f\in C^1(M)$ such that $f$ is constant outside a compact set,
$$|\nabla^sP_{s,t}f|^p_s\leq e^{-p\int_s^tK(r)\vd r}P_{s,t}|\nabla^tf|^p_t.$$
\item [$(x)$]
For any $0\leq s<u\leq t<T_c$ and positive function $f\in C^1(M)$ such that $f$ is constant outside a compact set,
$$\frac{(p\wedge 2)[P_{s,t}f^2-(P_{s,t}f^{2/(p\wedge 2)})^{p\wedge 2}]}{4(p\wedge 2-1)}\leq \int_s^te^{-2\int_u^tK(r)\vd r}\vd u\cdot P_{s,t}|\nabla^tf|^2_t.$$
When $p=1$, the inequality reduces to the log-Sobolev inequality
$$P_{s,t}(f^2\log f^2)-(P_{s,t}f^2)\log P_{s,t}f^2\leq 4\int_s^te^{-2\int_u^tK(r)\vd r}\vd u \cdot P_{s,t}|\nabla^tf|^2_t.$$
\end{enumerate}
\end{theorem}
\begin{proof}
First, by Theorems \ref{2t1}, \ref{coupling} and \ref{4t2}, the inequalities  (ii)--(x) can be derived from (i) by a similar discussion as in \cite[Theorem 4.3]{cheng} for the case without boundary.

Then, we assume (iv) and prove (i).  For a fixed point
$x\in M^{\circ}$, $t\in [0,T_c)$ and $X\in T_xM$, take $f\in C^{\infty}_0(M)$
such that $\nabla^tf=X$, ${\rm Hess}^t_{f}(x)=0$ and
$f=0$ in a neighborhood of $\partial M$. The argument in
\cite[Theorem 4.4]{cheng} works well for this case, i.e. $\mathcal{R}^Z_t\geq K(t)$ can be induced from
(iv).

So, it only leaves for us to derive ${\rm II}_t\geq 0$.
By Theorem \ref{II-form},  it is obvious to see that  we only need to consider the term with order  $\sqrt{t}$. So we do not need to care about the
terms, which  come from the time derivative about the metric, since they  at least have
order $t$. Therefore, by a similar procedure as in time-homogeneous
case (see \cite{WSe}). We conclude that  $\partial M$
is convex under the metric $g_t$ for all $t\in [0,T_c)$.
\end{proof}

\section{Appendix}
\begin{proof}[Proof of Theorem \ref{c2}]
  Without loss of generality, we also consider $s=0$ for simplicity.

  (a)
  By the It\^{o} formula, we have
 \begin{align*}
 \vd \phi_t^{-p}(X_t)&=\l<\nabla^t\phi_t^{-p}(X_t), u_t\vd B_t\r>_t+(L_t\phi_t^{-p}(X_t)+\partial_t\phi_t^{-p}(X_t))\vd t+N_t\phi_t^{-p}(X_t)\vd l_t\\
 &\leq \l<\nabla^t\phi_t^{-p}(X_t),u_t\vd B_t\r>_t-p\phi_t^{-p}(X_t)\{K_{\phi}^{(p)}(t)\vd t+N_t\log\phi_t(X_t)\vd l_t\}.
 \end{align*}
 So, $M_t:=\phi_t^{-p}(X_t)\exp{\l[p\int_0^tK_{\phi}^{(p)}(s)\vd s+p\int_0^tN_s\log \phi_s(X_s)\vd l_s\r]}$
 is a local martingale. Thus, using the Fatou lemma and noting that $\phi_t\geq 1$, we have
 \begin{align*}
& \mathbb{E}\l\{\phi_t^{-p}(X_t)\exp{\l[p\int_0^tK_{\phi}^{(p)}(s)\vd s+p\int_0^tN_s\log\phi_s(X_s)\vd l_s\r]}\r\}\\
 &\leq \liminf_{n\rightarrow\infty}\mathbb{E}^x\l\{\phi_t^{-p}(X_{t\wedge \zeta_n})\exp{\l[p\int_0^{t\wedge \zeta_n}K_{\phi}^{(p)}(s)\vd s+p\int_0^{t\wedge \zeta_n}N_s\log\phi_s(X_s)\vd l_s\r]}\r\}\\
 &\leq \phi_0^{-p}(x)\leq 1.
 \end{align*}
 Therefore,
 \begin{align}\label{eq0}\mathbb{E}^x\exp{\l[p\int_0^tN_s\log \phi_s(X_s)\vd l_s\r]}\leq \|\phi_t\|_{\infty}^{p}e^{-p\int_0^tK_{\phi}^{p}(s)\vd s},\ \ t\geq 0.\end{align}
 Since ${\rm II}_t\geq -N_t\log\phi_t$, by combining this with Theorem \ref{Bis} for $\sigma(t,\cdot)=-N_t\log\phi_t$ and Theorem \ref{P1} (i), we obtain
 \begin{align*}
 |\nabla^0P_{0,t}f|^p_0(x)\leq &(P_{0,t}|\nabla^tf|_t^{p/(p-1)})^{(p-1)}(x)\mathbb{E}^x\|Q_t\|^p\\
\leq & (P_{0,t}|\nabla^tf|^{p/(p-1)}_t(x))^{(p-1)}\mathbb{E}^x\exp{\bigg[p\int_0^tN_s\log\phi_s(X_s)\vd s\bigg]}\\
 \leq &
 \|\phi_t\|^p_{\infty}(P_{0,t}|\nabla^tf|_t^{p/(p-1)})^{(p-1)}(x)e^{-p\int_0^tK_{\phi}^p(s)\vd s}.
 \end{align*}
 Therefore, the first inequality holds.

 (b)  Let
 $$h(s)=\frac{\int_0^s\|\phi_u\|_{\infty}^{-2}e^{2\int_0^uK_{\phi}^{(2)}(r)\vd r}\vd u}{\int_0^t\|\phi_u\|_{\infty}^{-2}e^{2\int_0^uK_{\phi}^{(2)}(r)\vd r}\vd u},\ \quad s\in [0,t].$$
 Then the following inequality follows from the second formula in (\ref{2Bis}) and (\ref{eq0}) for $p=2$,
 \begin{align*}
 |\nabla^0P_{0,t}f|_0^2&\leq \frac{P_{0,t}f^2}{2}\mathbb{E}\int_0^th'(s)^2\|Q_s\|^2\vd s\\
 &\leq \frac{P_{0,t}f^2}{2}\mathbb{E}\int_0^th'(s)^2\|\phi_s\|_{\infty}^2\exp{\l[-2\int_0^sK_{\phi}^{(2)}(r)\vd r\r]}\vd s\\
 &\leq \frac{1}{2}\l[\int_0^t\|\phi_u\|_{\infty}^{-2}e^{2\int_0^uK_{\phi}^{(2)}(r)\vd r}\vd u\r]^{-1}P_{0,t}f^2.
 \end{align*}
 We complete the proof of \eqref{eq1}.
\end{proof}

\vspace{1cm}
\noindent\textbf{Acknowledgements}  \  The first author was supported in part by the starting-up research fund supplied by Zhejiang University of Technology (Grant No.109007329) and the Natural Science Foundation of Zhejiang University of Technology (Grant No. 2014XZ011).


\begin{thebibliography}{99}\setlength{\itemsep}{-1pt}
\bibitem{ACT} Arnaudon, M., Coulibaly, K., Thalmaier, A.:  Brownian motion with respect to a metric depending on time: definition, existence and applications to Ricci flow, C. R. Math. Acad. Sci. Paris {346},   773--778 (2008)

\bibitem{ACT1} Arnaudon, M., Coulibaly, K.,  Thalmaier, A.:  Horizontal diffusion in $C^1$ path space, S\'{e}minaire de Probabilit\'{e}s XLIII, 73--94, Lecture Notes in Math. (2006), Springer, Berlin (2011)


\bibitem{BCP}Bailesteanu, M., Cao, X., Pulemotov, A.: Gradient estimates for the heat equation under the Ricci flow, J. Funct. Anal. {258},   3517--3542 (2010)


\bibitem{Bakry} Bakry, D.: On Sobolev and logarithmic Sobolev inequalities for Markov semigroups, New trends in stochastic analysis (Charingworth, 1994), 43--75, World Sci. Publ., River Edge, NJ (1997)

\bibitem{BE}
Bakry, D., \'{E}mery, M.:  Hypercontractivit\'{e}de semi-groupes de diffusion, (French) [Hypercontractivity for diffusion semigroups] C. R. Acad. Sci. Paris S\'{e}r. I Math. {\bf 299},  775--778 (1984)

\bibitem{Bess} Besse, A. L.: Einstein manifolds, Ergebnisse der Mathematik und ihrer Grenzgebiete (3) [Results in Mathematics and Related Areas (3)], 10. Springer-Verlag, Berlin (1987)

\bibitem{Bismut} Bismut, J. -M.: Large deviations and the Malliavin calculus, Progress in Mathematics, 45. Birkh\"{a}user Boston, Inc., Boston, MA (1984)

\bibitem{Cr} Cranston, M.:
Gradient estimates on manifolds using coupling,
J. Funct. Anal. {99},   110--124 (1991)


\bibitem{CW94} {M.-F. Chen, F.-Y. Wang}, {Application of coupling method to the frist eigenvalue on manifold}, Sci. Sin. (A) 37:1 (1994), 1--14.

\bibitem{CW97a} {M.-F. Chen, F.-Y. Wang},  {General formula for lower bound of the first eigenvalue on Riemannian manifolds},  Sci. Sin. (A) 40:4 (1997), 384-394.

\bibitem{CW97b} {M.-F. Chen, F.-Y. Wang}, {Estimation of spectral gap for elliptic operators}, Transactions of the American Mathematical Society,  349:3 (1997), 1239-1267.

\bibitem{cheng} Cheng, L.-J.: Diffusion process on time-inhomogeneous manifolds, arXiv:math/1211.3621, preprint (2012)

\bibitem{Emery}\'{E}mery, M.: Stochastic Calculus in Manifolds, With an appendix by P.-A. Meyer. Universitext. Springer-Verlag, Berlin (1989)

\bibitem{EL} Elworthy, K. D., Li, X.-M.: Formulae for the derivatives of heat semigroups, J. Funct. Anal. {125},  252--286 (1994)

\bibitem{FWW} Fang, S., Wang, F.-Y., Wu, B.: Transportation-cost inequality on path spaces with uniform distance, Stochastic Process. Appl. {118},   2181--2197 (2008)

\bibitem{Fr} Friedman, A.:  Partial Differential Equations of Parabolic Type, Prentice-Hall, Inc., Englewood Cliffs, N.J. 1964 xiv+347 pp.

\bibitem{Gu} C. M. Guenther, The fundamental solution on manifolds
with time-dependent metrics,  J. Geo.  Anal.
12, 425-436 (2002).

\bibitem{Hsu} Hsu, E. P.:  Stochastic Analysis on Manifolds, Graduate Studies in Mathematics, 38. American Mathematical Society, Providence, RI (2002)

\bibitem{Hsu02}Hsu, E. P.:  Multiplicative functional for the heat equation on manifolds with boundary, Michigan Math. J. {50},  351--367 (2002)



\bibitem{Kasue1} Kasue, A.: A Laplacian comparison theorem and function theoretic properties of a complete Riemannian manifold. Japan. J. Math. (N.S.) {8},  309--341 (1982)

\bibitem{Kasue2} Kasue, A.: On Laplacian and Hessian comparison theorems, Proc. Japan Acad. Ser. A Math. Sci. {58}, 25--28 (1982)

\bibitem{Ku} Kuwada, K., Philipowski, R.: Non-explosion of diffusion processes on manifolds with time-dependent metric,  Math. Z. {268},  979--991 (2011)


\bibitem{Ku3}{Kuwada, K.}: Convergence of time-inhomogeneous geodesic random walks and its application to coupling methods, Ann. Probab. 40,  1945--1979 (2012)

\bibitem{KrR} Krylov, N. V., Rozovsk\v{i}, B. L.: Stochastic evolution equations, (Russian) Current problems in mathematics, Vol. {14}, 71--147, 256,  Akad. Nauk SSSR, Vsesoyuz. Inst. Nauchn. i Tekhn. Informatsii, Moscow (1979)

\bibitem{P10} Pulemotov, A.: Quasilinear parabolic equations and the Ricci flow on manifolds with boundary, arXiv: math/1012.2941, preprint (2010)

 \bibitem{Qian} Qian, Z.: A gradient estimate on a manifold with convex boundary, Proc. Roy. Soc. Edinburgh Sect. A {127},   171--179 (1997)

\bibitem{Shen}{Shen, Y.}: On Ricci deformation of a Riemannian metric on manifold with boundary, Pacific J. Math. {173},  203--221 (1996)


\bibitem{TW} Thalmaier, A., Wang, F.-Y.:  Gradient estimates for harmonic functions on regular domains in Riemannian manifolds,  J. Funct. Anal. {155},   109--124 (1998)


\bibitem{W04} Wang, F.-Y.: Equivalence of dimension-free Harnack inequality and curvature condition, Integral Equations Operator Theory {48},  547--552 (2004)


 \bibitem{W07} Wang, F.-Y.:  Estimates of the first Neumann eigenvalue and the log-Sobolev constant on non-convex manifolds,  Math. Nachr. {280},   1431--1439 (2007)

\bibitem{W09a} Wang, F.-Y.: Second fundamental form and gradient of Neumann semigroups, J. Funct. Anal. {256},  3461--3469 (2009)

\bibitem{WSe}  Wang, F.-Y.: Semigroup properties for the second fundamental form, Doc. Math. {15}, 527--543 (2010)

\bibitem{W10} Wang, F.-Y.: Harnack inequalities on manifolds with boundary and applications, J. Math. Pures Appl. (9) {94},  304--321 (2010)

\bibitem{WPC}  Wang, F.-Y.: Gradient and Harnack inequalities on noncompact manifolds with boundary, Pacific J. Math. {245},   185--200 (2010)

\bibitem{WAP} Wang, F.-Y.: Harnack inequality for SDE with multiplicative noise and extension to Neumann semigroup on nonconvex manifolds, Ann. Probab. {39},  1449--1467 (2011)

\bibitem{Wbook2} Wang, F.-Y.: Analysis for (Reflecting) Diffusion Processes on Riemannian Manifolds,  World Scientific, Singapore (2013)

\end{thebibliography}
\end{document}